\documentclass[twoside, 12pt, reqno]{amsart}

\DeclareSymbolFont{symbols}{OMS}{cmsy}{m}{n}

\usepackage{fullpage}
\usepackage[utf8]{inputenc}
\usepackage[T1]{fontenc}

\usepackage{mathrsfs}
\usepackage{mathtools}
\mathtoolsset{showonlyrefs,showmanualtags}
\usepackage{amsfonts}
\usepackage{amsmath, amssymb, amsthm}
\usepackage{enumerate}
\usepackage{graphicx, tikz}
\usepackage{ifthen}
\usetikzlibrary{arrows, scopes}
\usetikzlibrary{calc, cd}
\usetikzlibrary{decorations.pathreplacing}
\usetikzlibrary{decorations.pathmorphing}
\usetikzlibrary{patterns}
\usetikzlibrary{fadings}
\usetikzlibrary{cd}
\usetikzlibrary{automata, positioning}
\usetikzlibrary{calc}
\usepackage{tikz-cd}
\usepackage{stmaryrd}
\usepackage{setspace}
\usepackage{hyperref}
\usepackage[alphabetic, initials]{amsrefs}
\usepackage{bbm}

\newtheorem{thm}{Theorem}[section]
\newtheorem{lem}[thm]{Lemma}

\newtheorem{prop}[thm]{Proposition}
\theoremstyle{definition}
\newtheorem{defi}[thm]{Definition}
\newtheorem{example}[thm]{Example}
\theoremstyle{remark}
\newtheorem{rmk}[thm]{Remark}

\newcommand{\Hil}{\mathcal{H}}
\newcommand{\slot}{{~\cdot~}}

\newcommand{\U}{{\mathsf{U}}}

\newcommand{\cC}{\mathcal{C}}

\newcommand{\K}{\mathcal{K}}

\newcommand{\cD}{\mathcal{D}}

\newcommand{\cK}{\mathcal{K}}
\newcommand{\A}{\mathcal{A}}

\newcommand{\RR}{\mathbb{R}}
\newcommand{\TT}{\mathbb{T}}

\newcommand{\CC}{\mathbb{C}}

\newcommand{\NN}{\mathbb{N}}

\DeclareMathOperator{\End}{End}

\DeclareMathOperator{\Rep}{Rep}

\DeclareMathOperator{\Hom}{Hom}

\DeclareMathOperator{\Mor}{Mor}

\DeclareMathOperator{\ev}{ev}

\DeclareMathOperator{\Meas}{Meas}

\DeclareMathOperator{\Ad}{Ad}

\DeclareMathOperator{\id}{id}

\DeclareMathOperator{\B}{B}

\DeclareMathOperator{\Aut}{Aut}
\DeclareMathOperator{\Amp}{Amp}

\newcommand{\dd}{\mathrm{d}}

\usepackage{xspace}

\DeclareRobustCommand{\ie}{i.e.\@\xspace}
\makeatletter
\DeclareRobustCommand{\etc}{%
    \@ifnextchar{.}%
        {etc}%
        {etc.\@\xspace}%
}
\makeatother

\def\u1net{{\A_\RR}}

\DeclareMathOperator{\Out}{\mathsf{Out}}

\newcommand{\LL}{\mathcal{L}}

\usepackage{enumitem}

\newcommand{\concrete}{\mathrm{concrete}}

\newcommand{\Borel}{\mathrm{Borel}}

\newcommand{\Bim}{\mathrm{Bim}}

\begin{document}
\date{\today}
\dateposted{\today}
\newcommand{\mytitle}{%
    Continuous categories of endomorphisms associated with $G$-kernels}
\title{\mytitle}
\author{Marcel Bischoff}
\email{mrclbschff@gmail.com}
\author{Pradyut Karmakar}
\address{Sam Houston State University, 332 G LDB, 1900 Avenue I, Huntsville, TX 77340}
\email{karmakar.pradyut@gmail.com}
\dedicatory{This article is dedicated to the memory of the second author's mother}

\begin{abstract}
    We generalize the construction of tensor categories of endomorphisms of a type III factor $M$ associated with a $G$-kernel, from the case of a discrete group $G$ to that of a compact second countable group. Our approach is based on the construction of a unitary tensor functor from a category of $C(G)$-modules to the category of endomorphisms of $M$. This functor maps a $C(G)$-module, realized as the space of square-integrable functions on a measure space, to a continuous family of endomorphisms of $M$. The resulting structure is a continuous category of endomorphisms, which provides a new framework for studying the interplay between subfactor theory and the representation theory of continuous groups.
\end{abstract}

\maketitle

\setcounter{tocdepth}{3}
\tableofcontents

\section{Introduction}
\label{sec:Introduction} 
The direct sum theory of sectors associated with a $G$-kernel for a finite group has been examined in the work of Izumi in \cite{Iz2015}.
However, the literature lacks a formal notion of the \emph{direct integral theory of sectors}. This article introduces a construction that develops the direct integral theory for sectors associated with a $G$-kernel, where $G$ denotes a compact second countable group.
The main goal of this paper is to generalize a well-known construction of endomorphisms of a type III factor $M$ from discrete groups to the case of a compact second countable group $G$.

Let us first recall the discrete case. Let $G$ be a discrete group. A $G$-kernel on a type III factor $M$ is a homomorphism $\theta \colon G \to \Out(M)$ with a certain obstruction $[\omega]\in H^3_{\mathrm{Borel}}(G, \mathrm{U}(1))$; we refer readers to \cite{Jo1980}, \cite{Su1980-II} for the definition of $G$-kernels. More precisely, the following is well-known to experts.
\begin{prop}
    Let $G$ be a group and $\theta \colon G \to \Out(M)$ be a $G$-kernel on a type III
    factor $M$ with obstruction $[\omega]\in H^3(G, \mathrm{U}(1))$.

    Then there is a fully faithful dagger-unitary tensor functor $F$ from the category of
    finite-dimensional $G$-graded Hilbert spaces twisted by $\omega$, namely $\mathrm{Hilb}^{\omega}_{G}$, to
    the category $\End(M)$ with $[F(\mathbb{C}^g)] = \theta_g$.
\end{prop}
Let $\mathbb{C}^n_g$ be the $G$-graded Hilbert space with grading given by a function
$g\colon I_n \to G$.
Then $F(\mathbb{C}^n_g) =\sum_{i=1}^n \Ad v_i\circ \alpha_{g(i)}$, where
$\{v_i\in M\}_{i\in I_n}$ are
generators of the Cuntz algebra $\mathcal{O}_n$.

The goal of this paper is to generalize this construction to the continuous case
where $G$ is a compact Hausdorff second countable group.

To facilitate this generalization, we first re-examine some fundamental constructions.
Denote $I_n= \{1, \ldots, n\}$ and $I_\infty =\NN$.
\begin{prop}
    \label{prop:UnitaryCuntz}
    There is a one-to-one correspondence between right $M$-modular unitary maps 
    $U\colon L^2M_M\otimes \ell^2I_n\to L^2M_M$ and 
    $\{v_i\in M\}_{i\in I_n}$ representations of the Cuntz algebra 
    $\mathcal{O}_n$
    in $M$.
\end{prop}
The following corresponds to the action of an endomorphism of $M$ on 
a Cuntz algebra realized in $M$.
\begin{prop}
    Let $\rho\in \End(M)$. 
    Consider the W${}^*$-category $\Amp(M)$
    with objects $L^2M_M$, $L^2M_M\otimes \cK$ where $\cK$ is a separable 
    Hilbert space and morphisms are bounded right $M$-modular maps
    denoted by $\Hom_{-M}(\slot, \slot)$.

    Then there is a unitary endofunctor on $\operatorname{Amp}(M)$
    acting trivially on objects and on a morphism
    $m\otimes t\in \Hom_{-M}(L^2M_M\otimes \Hil, L^2M_M \otimes \cK)$
    for some $m\in M$ and $t\in \B(\Hil, \cK)$ by
    $$
    {}^\rho(m\otimes t) =\rho(m) \otimes t ,\qquad \text{where } m\otimes t \text{ is a simple tensor in $M \otimes \B(\Hil, \cK)$.}
    $$
\end{prop}
\begin{rmk}
    Here, $\Hom_{-M}(\mathcal{A}_M, \mathcal{B}_M)$ denotes the space of right $M$-modular maps from $\mathcal{A}_M$ to $\mathcal{B}_M$.
\end{rmk}
The action ${}^\rho(\slot)$ generalizes endomorphisms on Cuntz algebras in the following
sense.
Let $(v_i)_i$ be a family of generators of a Cuntz algebra in $M$ and let $U_{(v_i)_i}$ be the associated
unitary map from Proposition~\ref{prop:UnitaryCuntz} above.
Then ${}^\rho U_{(v_i)_i} = U_{(\rho(v_i))_i}$.

We now move to the continuous case. Let $G$ be a compact second countable group. We consider the W${}^*$-category $\Rep(C(G))$ of $C(G)$-modules that are separable Hilbert spaces.
Let $X=(X, \mu_X, g_X)$ be a triple consisting of a topological space $X$, a finite measure $\mu_X$, and a Borel map $g_X\colon X\to G$.

Then $L^2X$ becomes a $C(G)$-module via the action $f.\xi = M_{f\circ g_X}\xi$, where $f \in C(G)$ and $\xi \in L^2X$.

Consider the projection-valued measure $E_X=E_{\mu_X}$ characterized by
$$M_{f\circ g_X} = \int_X f(g_X(x)) \,\dd E_X(x)$$ for $f\in C(G)$.
Define 
\begin{equation}
    \sigma_X(m) = \int_X \alpha_{g_X(x)}(m)\otimes \dd E_{X}(x)\,.
\end{equation}
Then $\sigma_X\in \Mor(M, M \otimes \B(L^2X))$. 
Because $M$ is type III we can choose a unitary 
$U_X\in \U_{-M}(L^2M_M\otimes L^2X , L^2M_M)$ (unitary right $M$-modular map) and define 
$\rho_X = \Ad U_X \circ \sigma_X \in \End(M)$.

This construction gives us a functor from a category of measure spaces to the category of endomorphisms of $M$.
Recall we call a $C(G)$-module concrete if it is of the form $L^2X$ for a triple
$(X, \mu_X, g_X)$.
\begin{prop}
    \label{prop:UnitaryFunctor}
    The assignment $X \mapsto \rho_X$ defines a unitary functor 
    from the full subcategory of concrete $C(G)$-modules (of the form $L^2X$) 
    to the category $\End(M)$.
\end{prop}

\begin{proof}
    We first need to define the morphisms in the source category. Let $X_1$ and $X_2$ be two objects.
    A morphism $t\colon X_1 \to X_2$ is given by a bounded $C(G)$-equivariant
    linear map $t\colon L^2X_1 \to L^2X_2$.

    The functor maps this morphism to an intertwiner $F_t$ such that $F_t \rho_{X_1}(m) = \rho_{X_2}(m) F_t$
    for all $m \in M$, given by $F_t =U_{X_2}(1_M \otimes t) U_{X_1}^\ast$.

    It is straightforward to check that this defines a unitary functor.
\end{proof}
\begin{rmk}\label{uniqueness}
Given a fixed $G$-kernel $(M, \theta)$, it is worth noting that our functor is unique up to natural transformation and does not depend on the choice of unitaries.
\end{rmk}

The structure of the paper is as follows. Section ~\ref{sec:preliminaries} provides background on W${}^{\ast}$ and C${}^{\ast}$-tensor categories, recalls the definition of the $G$-kernel, and discusses the topological structure on morphisms in the category $\Amp(M)$. Section~\ref{sec:Tensoriality} summarizes how manipulation of the integral equation derived from the $G$-kernel establishes the tensoriality of the functor; further details are provided in the Appendix. Section~\ref{sec:Cuntz} demonstrates that the equivalence of Cuntz algebras within the type III factor $M$ corresponds to absorbing unitaries, specifically that the right $M$ modular unitaries absorb when $M$ is amplified by bounded operators on a separable Hilbert space. Section~\ref{sec:Cuntz functor} introduces the functor acting on absorbing unitaries, which is extensively used in the Appendix section to demonstrate that the Pentagram diagram commutes. The construction of this endofunctor is expected, as it facilitates the composition of endomorphisms. In Section~\ref{sec:minimal}, we introduce the notion of minimal $G$-kernels and show that if the faithful functor $F\colon \Rep_{\concrete}(C(G)) \to \End(M)$ is full, then the $G$-kernel is necessarily minimal. In section~\ref{sec:Tensor structure}, we give the tensor structure on the category of endomorphisms under the image of the functor $F$ and construct the associator using the $G$-kernel.
\section{Acknowledgment}
The second author expresses sincere thanks to Yasuyuki Kawahigashi for many helpful discussions, and to Luca Giorgetti for carefully providing feedback on an earlier draft.

\section{Preliminaries}
\label{sec:preliminaries}
Throughout the article, we regard $M \subseteq \B(L^2(M))$, where $L^2(M)$ is the canonical Haagerup $L^2$ space \cite{Ha1975}.
We denote \emph{monoidal categories} by  $(\cC,\otimes,\mathbf{1},a,\iota)$,  where $\cC$ is a category, $\otimes \colon \cC \times \cC \to \cC$ is the tensor product bifunctor, $a$ is the associativity isomorphism, $\mathbf{1}$ is the unit object of $\cC$, and $\iota \colon \mathbf{1} \otimes \mathbf{1} \to \mathbf{1}$ is the unit isomorphism.
 These data satisfy the pentagon and the unit axioms (see \cite[Definition~2.1.1]{EtGeNiOs2015}).
We denote the morphism space by $\Hom(X, Y)$ for any pair of objects $X, Y \in \cC$ and $\End(X):=\Hom(X,X)$ for a single object $X$.
 
In order to build the connection between $\End(M)$ and a subcategory of $\Rep(C(X))$ for some compact Hausdorff space $X$, we need analytic structure in addition to monoidal structure, and W${}^{\ast}$-tensor categories provide exactly this.

We record some basic details about W${}^{\ast}$-categories here.
The definitions are standard and more details can be found in \cite{EtGeNiOs2015}, \cite{GhLiRo1985} and \cite{HeNiPe2024}.

\subsection{W*-categories}
To define a W${}^{\ast}$-category, we begin by recalling the definition of a C${}^{\ast}$-category.
Let $\cC$ be a category with objects $X,Y,\ldots$.  
For any two objects $X,Y$, write $\Hom(X,Y)$ for the set of morphisms from $X$ to $Y$.  
The category $\cC$ is called a \emph{C${}^{\ast}$-category} if the following axioms hold:

\begin{enumerate}
    \item $\Hom(X,Y)$ is a Banach space and composition
    \begin{align*}
    \Hom(Y,Z)\times \Hom(X,Y)\to \Hom(X,Z)
    \end{align*}
is bilinear.

    \item For all $y\in \Hom(Y,Z)$ and $x\in \Hom(X,Y)$, the norm inequality holds, namely
    \begin{align*}
    \|yx\| \leq \|y\|\|x\|\,.
    \end{align*}
    \item $\cC$ is a $*$-category, \ie, for each pair of objects $X, Y$ is equipped with a map  
   \begin{align*}
    x\in \Hom(X,Y)\mapsto x^{*}\in \Hom(Y,X)
  \end{align*}
    satisfying
    \begin{enumerate}
        \item $(cx+y)^{*}=\overline{c}\,x^{*}+y^{*}$,
        \item $x^{**}=x$,
        \item $(xy)^{*}=y^{*}x^{*}$.
    \end{enumerate}

    \item The $C^{*}$-identity holds:
   \begin{align*}
    \|x^{*}x\|=\|x\|^{2}\qquad\text{for all}\,x\in \Hom(X,Y).
  \end{align*}
   In particular, $\Hom(X,X)$ is a $C^{*}$-algebra.
\item For every $x\in \Hom(X,Y)$, the element $x^{*}x$ is positive in the
    $C^{\ast}$-algebra $\Hom(X,X)$.
\end{enumerate}

Now let $\cC$ and $\cD$ be two C${}^{\ast}$-categories.  
A functor $F \colon \cC\to\cD$ is called a C${}^{\ast}$-functor if:
\begin{enumerate}
    \item $F \colon \Hom(X,Y)\to \Hom(F(X),F(Y))$ is linear,
    \item $F$ is $*$-preserving, i.e.
    \begin{align*}
    F(x^{*}) = F(x)^{*}\qquad\text{for all }x\in \Hom(X,Y)\,.
    \end{align*}
\end{enumerate}

A category $\cC$ is called
a \emph{W${}^\ast$-category} if it is a C${}^\ast$-category
such that for each object $X, Y \in \cC$, the Banach space of morphisms $\Hom(X, Y)$ has a predual. Furthermore, if $\cC$ is a W${}^\ast$-category and $X \in \cC$ is an object then $\Hom(X, X)$ is a von Neumann algebra.

For a $C^{\ast}$-algebra $X$, $X_{*}$ and $X^{*}$ denote the predual and dual of $X$, respectively.
A $W^{\ast}$-functor is the same as a $C^{\ast}$-functor, with the additional structure that 
the functor 
\begin{align*}
    F \colon \Hom(X,Y) \rightarrow \Hom(F(X),F(Y))
\end{align*}
admits a predual. By this, we mean the following. Let $A_{X,Y}$ denote the predual of $\Hom(X,Y)$ and 
let $B_{F(X),F(Y)}$ denote the predual of $\Hom(F(X),F(Y))$. Suppose
\begin{align*}
    i \colon \Hom(X,Y) \to A_{X,Y}^{\ast}, 
    \qquad
    j \colon \Hom(F(X),F(Y)) \to B_{F(X),F(Y)}^{\ast},
\end{align*}
where the maps $i$ and $j$ exist. Then there is a morphism
\begin{align*}
    F_{\ast} \colon B_{F(X),F(Y)}  \rightarrow A_{X,Y}
\end{align*}
such that
\begin{align*}
    \langle j (F(T)), b \rangle= \langle i(T), F_{\ast}(b) \rangle \quad \text{ for all }T \in \Hom(X,Y), b \in B_{F(X), F(Y)}\,,
\end{align*}
where $\langle \cdot, \cdot \rangle$ is the duality pairing.

A functor $F \colon \cC \to \cD$ is said to be \emph{faithful} if, for every pair of objects $X, Y \in \cC$, the map
\begin{align*}
\Hom(X,Y) \to \Hom(F(X), F(Y))
\end{align*}
is injective.
Furthermore, a functor $F$ is said to be \emph{full} if, for every pair of objects $X, Y \in \cC$, the map
\begin{align*}
\Hom(X,Y) \to \Hom(F(X), F(Y))
\end{align*}
is surjective, and $F$ is \emph{fully faithful} if $F$ is full and faithful.

Let $(\cC,\otimes,\mathbf{1},a^{\cC},\iota^{\cC})$
 and 
$(\cD,\otimes,\mathbf{1},a^{\cD},\iota^{\cD})$
be monoidal categories.
A \emph{monoidal functor} from $\cC$ to $\cD$ consists of a functor
$F \colon \cC \to \cD$ together with a natural isomorphism

\begin{align*}
J_{X,Y} \colon F(X)\otimes F(Y) \rightarrow F(X\otimes Y),
\end{align*}
defined for all $X,Y\in \cC$, and such that
$F(\mathbf{1})\cong \mathbf{1}$ in $\cD$.
These data must satisfy the coherence condition that, for all
$X,Y,Z \in \cC$, the diagram

\[
\begin{tikzcd}[column sep=large, row sep=large]
(F(X)\otimes F(Y))\otimes F(Z)
  \arrow[r,"J_{X,Y}\otimes \mathrm{id}_{F(Z)}"]
  \arrow[d,"a^{\mathcal{D}}_{F(X),F(Y),F(Z)}"']
&
F(X\otimes Y)\otimes F(Z)
  \arrow[r,"J_{X\otimes Y,\, Z}"]
&
F((X\otimes Y)\otimes Z)
  \arrow[d,"F(a^{\mathcal{C}}_{X,Y,Z})"]
\\
F(X)\otimes(F(Y)\otimes F(Z))
  \arrow[r,"\mathrm{id}_{F(X)}\otimes J_{Y,Z}"']
&
F(X)\otimes F(Y\otimes Z)
  \arrow[r,"J_{X,\, Y\otimes Z}"']
&
F(X\otimes (Y\otimes Z))
\end{tikzcd}
\]
commutes.

A monoidal functor $F$ is an \emph{equivalence of monoidal categories}
if the underlying functor $F \colon \cC \to \cD$ is an equivalence of
categories.

Recall that a \emph{monoidal natural transformation} between two monoidal functors $F, G \colon \cC \to \cD$ is a family of morphisms $\eta_X \colon F(X) \to G(X)$, natural in $X$, such that the following diagram commutes:
$$
\begin{tikzcd}
F(X) \otimes F(Y) 
  \arrow[rrr, "J^{F}_{X,Y}"] 
  \arrow[d, "\eta_X \otimes \eta_Y"'] 
& & & 
F(X \otimes Y) 
  \arrow[d, "\eta_{X\otimes Y}"] 
\\
G(X) \otimes G(Y) 
  \arrow[rrr, "J^{G}_{X,Y}"'] 
& & & 
G(X \otimes Y)
\end{tikzcd}
$$
where $F$ and $G$ are functors from $\cC$ to $\cD$.

Moreover, a \emph{C${}^{\ast}$-tensor category} is a C${}^{\ast}$-category equipped with a monoidal structure such that the tensor product is compatible with the $C{}^{\ast}$-operations, and a \emph{W${}^{\ast}$-tensor category} is exactly a
monoidal W${}^{\ast}$-category
with the additional requirement that the tensor product is normal (ultraweakly continuous) in each variable.

Examples of W${}^{\ast}$-tensor categories include the category of bimodules over a von Neumann algebra $M$, denoted by $\Bim(M)$, where objects consist of separable Hilbert spaces with left and right actions of $M$, morphisms are $M$-bimodular bounded maps, the tensor product structure is given by Connes fusion, and the unit object is the standard bimodule $L^2(M)$.
Another important example arising from algebraic quantum field theory is the \emph{DHR} superselection sectors (\cite{DoHaRo1969II}), which form a W${}^{\ast}$-tensor category. 
The objects consist of localized endomorphisms of the observable algebra, while the morphisms are intertwiners.
The tensor product structure is given by composition of localized sectors.
The most natural examples of monoidal functors encountered in conformal field theory are $\alpha$-induction and $\sigma$-restriction (details can be found in \cite{BiKaLoRe2014-2}).

We recall the definition of $G$-kernel here:
\begin{defi}\label{kernl}
Let $G$ be a second countable compact group and $M$ a type $\mathrm{III}$ factor. A monomorphism 
$\theta \colon G \to \Out(M)$ is said to be a $G$-kernel provided that there exists 
a Borel map $\alpha \colon G \to \Aut(M)$ satisfying
\begin{align*}
\varepsilon \circ \alpha = \theta,
\end{align*}

where $\varepsilon \colon \Aut(M) \to \Out(M)$ denotes the quotient map 
sending an automorphism $\alpha$ to its outer class $[\alpha]=\{\Ad v \circ \alpha : v \in \U(M)\}$, where $\U(M)$ is the unitary group of $M$.
\end{defi}

\subsection{Topology on morphisms in $\Amp(M)$}
Consider the W${}^{\ast}$-category of right $M$-modules, \ie pairs $\Hil_{M}=(\Hil, \rho)$ of a Hilbert space $\Hil$ with a right action $\rho$ of $M$ on $\Hil$.
Let $\Amp(M)$ be the full subcategory whose objects are $L^2M_M$ or $L^2M_M \otimes \Hil$.

Observe that $\Hom_{-M}(L^2M_M \otimes \Hil, L^2M_M \otimes \K)$ is the weak${}^{\ast}$-closure of the algebraic tensor product $M \odot \B(\Hil, \K)$ and the weak${}^{\ast}$-topology is the subspace topology of the predual of the von Neumann algebra
\begin{align*}
\End_{-M}((L^2M_M \otimes \Hil) \oplus (L^2M_M \otimes \K))=\End_{-M}((L^2M_M \otimes (\Hil \oplus \K)))
\end{align*}
via the canonical embedding.
We denote $\Hom_{-M}(L^2M_M \otimes \Hil, L^2M_M \otimes \K)=M \otimes \B(\Hil, \K)$, realizing that this is the spatial product of W${}^{\ast}$-algebras.
\begin{rmk}
Observe that $\B(\Hil, \K)$ sits as a subspace of $\B(\Hil \oplus \K)$ via the canonical embedding
\begin{align*}
\B(\Hil, \K) \ni t \mapsto V_{\K} t V_{\Hil}^{\ast} \in \B(\Hil \oplus \K)\,,
\end{align*}
where $V_{\K} \colon \K \to \Hil \oplus \K$ and $V_{\Hil} \colon \Hil \to \Hil \oplus \K$ are the canonical isometries.
Consequently, $t_n \to t$ in weak operator topology in $\B(\Hil, \K)$ if and only if for any $h\in \Hil$, $k \in \K$, one has $$\langle t_n h, k \rangle_{\K} \to \langle t h, k \rangle_{\K}.$$
Therefore, we say $x_n \to x$ in $M \otimes \B(\Hil, \K)$ in the weak operator topology if and only if, for any $\xi \in L^2(M)$, $h \in \Hil$ and $k \in \K$, one has
\begin{align*}
(\omega_{\xi, \xi} \otimes \omega_{h,k}) (x_n) \to (\omega_{\xi, \xi} \otimes \omega_{h,k}) (x)\,.
\end{align*}
Here $\omega_{\cdot, \cdot}$ denotes the vector state.
Therefore a sequence $x_n \to x$ in the weak${}^{\ast}$-topology if and only if $x_n \to x$ in the weak operator topology on bounded sets.
\end{rmk}

\section{Tensoriality}
\label{sec:Tensoriality}

Throughout this section, $X$ (and $Y$, $Z$) denotes an object in the category of measure spaces under consideration.
In this section, we show that the functor $F \colon X \mapsto \rho_X$ defined in the previous section is a unitary tensor functor. This means that it preserves the tensor product structure of the source category. The tensor product in the source category is given by the cartesian product of the underlying measure spaces. The tensor product in the target category is the composition of endomorphisms.

We first define the concept of an intertwiner in our setting.
\begin{defi}
    Let $\mathcal{A}_M, \mathcal{B}_M$ be objects in $\operatorname{Amp}(M)$. Let $\alpha\colon M\to \End_{-M}(\mathcal{A}_M)$ and $\beta \colon M\to \End_{-M}(\mathcal{B}_M)$ be normal $\ast$-morphisms. The space of intertwiners between $\alpha$ and $\beta$ is defined as
    \begin{equation}
        (\alpha, \beta) = \{T\in \Hom_{-M}(\mathcal{A}_M, \mathcal{B}_M): T\alpha(m) = \beta(m)T
        \text{ for all } m\in M\}\,.
    \end{equation}
    Here $\End_{-M}(\mathcal{A}_M)$ denotes the algebra of endomorphisms of $\mathcal{A}_M$ that commute with the right $M$-module structure.
\end{defi}

This definition is a generalization of the usual notion of intertwiners. For instance, if we consider two endomorphisms $\rho, \sigma \in \End(M)$, an intertwiner between them is an operator $T \in M$ such that $T\rho(m) = \sigma(m)T$ for all $m\in M$. In our notation, this corresponds to the case where $\mathcal{A}_M = \mathcal{B}_M = L^2(M)$ and $T \in (\rho, \sigma)$.

Our goal is to show that for any two objects $X, Y$ in our source category, there is a natural unitary isomorphism between $\rho_X \rho_Y$ and $\rho_{X\times Y}$. This is the content of the following proposition.
\begin{prop}
The functor $X \mapsto \rho_X$ is a unitary tensor functor.
\end{prop}

\begin{proof}
The proof proceeds by constructing a natural unitary intertwiner
$W_{X,Y} \in (\rho_X\rho_Y, \rho_{X \times Y})$.
Recall that $\rho_X = \Ad U_X \circ \sigma_X$. This means that $U_X$ is a unitary in $(\sigma_X, \rho_X)$.
The core of the proof lies in the analysis of the following diagram:

$$
\begin{tikzcd}
\rho_X\rho_Y \arrow[dd, rightarrow, "W_{X,Y}"'] \arrow[rrr, no head, Rightarrow, "(U_X^{\ast} \otimes 1_Y) {}^{\rho_X}U_Y^{\ast}"] &                                                                                                                          &                                                                                           & (\sigma_X\otimes \id_Y)\sigma_Y \arrow[dd, rightarrow, "W_{\sigma_{X,Y}}"] \\
\\ \rho_{X\times Y} \arrow[rrr, no head, Rightarrow, "U_{X\times Y}"'] & & & \sigma_{X\times Y}
\end{tikzcd}
$$
The diagram shows that the natural unitary isomorphism $W_{X,Y}$ is composed of three parts: $W_{X,Y} = U_{X\times Y}\, W_{\sigma_{X,Y}}\, (U_X^{\ast} \otimes 1_Y)\, ({}^{\rho_X}U_Y)^{\ast}$, where $W_{\sigma_{X,Y}} = (\id_{L^2(M)} \otimes \iota_{X,Y}^{\ast}) u_{X,Y}^{\ast}$ is the intertwiner from $(\sigma_X \otimes \id_Y)\sigma_Y$ to $\sigma_{X\times Y}$ (constructed in the appendix). This matches the explicit formula for $W_{X,Y}$ given in~\eqref{tensgg}.

The construction of $W_{\sigma_{X,Y}}$ relies on the associativity of the fusion of defects in the underlying field theory. The calculation involves a careful manipulation of the integral expressions for the $\sigma$ morphisms. The key identity is the following:
\begin{align}
    \nonumber &\int_{X\times Y\times Z}
    u_{g_X(x)g_Y(y), g_Z(z)}u_{g_X(x),g_Y(y)} \otimes 
    \dd E_X(x)\otimes \dd E_Y(y) \otimes \dd E_Z(z) \\
    &=
    \int_{X\times Y\times Z}
    \omega_{g_X(x), g_Y(y), g_Z(z)} 
    u_{g_X(x),g_Y(y) g_Z(z)}\alpha_{g_X(x)}(u_{g_Y(y),g_Z(z)}) \otimes 
    \dd E_X(x)\otimes \dd E_Y(y) \otimes \dd E_Z(z) \,.
\end{align}
This identity, where $\omega$ is the 3-cocycle associated with the $G$-kernel, allows us to relate the two different ways of composing the tensor products, as shown in the following commutative diagram.
The diagram below illustrates the associativity of the tensor product. The two paths from the top left to the bottom right corner correspond to the two ways of parenthesizing the tensor product of three objects $X, Y, Z$, and the computation shows these paths are related by a unitary isomorphism involving the 3-cocycle $\omega$.

$$
\begin{tikzcd}
(\rho_X\rho_Y)\rho_Z 
  \arrow[dd, leftrightarrow]
  \arrow[rrr, equal]
& & &
\rho_X(\rho_Y\rho_Z)
  \arrow[dd, leftrightarrow]
\\
\\
\rho_{X\times Y}\rho_Z
  \arrow[dd, leftrightarrow]
& & &
\rho_X\rho_{Y\times Z}
  \arrow[dd, leftrightarrow]
\\
\\
\rho_{(X\times Y)\times Z}
  \arrow[rrr, equal]
& & &
\rho_{X\times (Y\times Z)}
\end{tikzcd}
$$

\end{proof}
The details are discussed in the appendix section.
Let $X=(X, \mu_X, g_X)$ as before. We use shorthand notation $\id_X := \id_{\B(L^2X)}$
and $1_X: = 1_{\B(L^2X)} = \id_{L^2X}$ for unit morphism and unit operator, respectively.
From
\begin{align}
    &\int_{X\times Y\times Z}
    u_{g_X(x)g_Y(y), g_Z(z)}u_{g_X(x),g_Y(y)} \otimes 
    \dd E_X(x)\otimes \dd E_Y(y) \otimes \dd E_Z(z) \\
    &=
    \int_{X\times Y\times Z}
    \omega_{g_X(x), g_Y(y), g_Z(z)} 
    u_{g_X(x),g_Y(y) g_Z(z)}\alpha_{g_X(x)}(u_{g_Y(y),g_Z(z)}) \otimes 
    \dd E_X(x)\otimes \dd E_Y(y) \otimes \dd E_Z(z) 
\end{align}
where $\omega$ is a cocycle in $Z^3_{\Borel}(G, \U(1))$,
it follows that
\begin{align}
    & ((\sigma_X\otimes \id_Y)\otimes\id_Z)(\sigma_Y\otimes\id_Z)\sigma_Z \\
     & \xrightarrow{\sim}
    \int\limits_{X\times Y\times Z}
    \alpha_{g_X(x)}\alpha_{g_Y(y)}\alpha_{g_Z(z)}(\slot)
    \otimes \dd E_X(x)\otimes \dd E_Y(y) \otimes \dd E_Z(z) \\
    & \xrightarrow{u_{X, Y}^{\ast} \otimes 1_Z}
    \int\limits_{X\times Y\times Z}
    \alpha_{g_X(x)g_Y(y)}\alpha_{g_Z(z)} (\slot)
    \otimes \dd E_{X}(x) \otimes \dd E_Y(y) \otimes \dd E_Z(z) \\
    & \xrightarrow{u_{X\times Y, Z}^{\ast}}
    \int\limits_{X\times Y\times Z}
    \alpha_{g_X(x)g_Y(y)g_Z(z)} (\slot)
    \otimes \dd E_{(X\times Y)\times Z}((x,y), z) \\
\end{align}
equals
\begin{align}
    & ((\sigma_X\otimes \id_Y)\otimes\id_Z)(\sigma_Y\otimes\id_Z)\sigma_Z \\
     & \xrightarrow{\sim}
    \int\limits_{X\times Y\times Z}
    \alpha_{g_X(x)}\alpha_{g_Y(y)}\alpha_{g_Z(z)}(\slot)
    \otimes \dd E_X(x)\otimes \dd E_Y(y) \otimes \dd E_Z(z) \\
    & \xrightarrow{{}^{\sigma_X}u_{Y, Z}^{\ast}}
    \int\limits_{X\times Y\times Z}
    \alpha_{g_X(x)}\alpha_{g_Y(y)g_Z(z)} (\slot)
    \otimes \dd E_X(x) \otimes \dd E_Y(y) \otimes \dd E_Z(z)\\
    & \xrightarrow{u_{X, Y \times Z}^{\ast}}
    \int\limits_{X\times Y\times Z}
    \alpha_{g_X(x)g_Y(y)g_Z(z)} (\slot)
    \otimes \dd E_{X\times (Y\times Z)}(x,(y, z))\\
    & \xrightarrow{\tilde{\omega}_{X,Y,Z}^{\ast}}
    \int\limits_{X\times Y\times Z}
    \alpha_{g_X(x)g_Y(y)g_Z(z)} (\slot)
    \otimes \dd E_{(X\times Y)\times Z}((x,y), z)\,.
\end{align}

\section{Cuntz algebras and absorbing unitaries}
\label{sec:Cuntz}
Denote $I_n= \{1, \ldots, n\}$ for $n \in \NN$ and $I_\infty =\NN$.
Let $r_{L^2M_M} \colon L^2M_M \otimes \CC \to L^2M_M$ be the canonical isomorphism defined by $\xi \otimes c \mapsto c \xi$.
\begin{lem}\label{convseri}
Let $\{v_i: i \in I\}$ be a family generating the Cuntz algebra inside the von Neumann algebra $M \subseteq \B(L^2(M))$ and let $\ev_i:=\langle e_i| \in \B(\ell^2(I), \CC)$ where $\{e_i: i \in I\}$ is an orthonormal basis of $\ell^2(I)$, where $I=I_n$ or $I=I_\infty$. Then 
\begin{align*}
\sum_i v_i \otimes \ev_i
\end{align*}
converges in $\Hom_{-M}(L^2M_M \otimes \ell^2(I), L^2M_M \otimes \CC)$. 
\end{lem}
\begin{proof}
Since $\Hom_{-M}(L^2M_M \otimes \ell^2(I), L^2M_M \otimes \CC) \cong M \otimes \B(\ell^2(I), \CC)$
(spatial tensor product), it is enough to check convergence in $M \otimes \B(\ell^2(I), \CC)$ with respect to the weak${}^{\ast}$-topology.
Let $\phi$ be in the predual of $M \otimes \B(\ell^2(I), \CC)$; then
$$\phi=\sum_i \phi_i \otimes e_i$$ with $\phi_i \in M_{\ast}$.
Then $\|\phi\|=\sum_i \|\phi_i\|$.
Note that
\begin{align*}
\phi(m \otimes t)=\sum_i \phi_i(m) \otimes t e_i \qquad m \in M,~ t \in \B(\ell^2(I), \CC)\,.
\end{align*}
Thus
\begin{align*}
\left|\phi\left(\sum_{i=n}^{m} v_i \otimes \ev_i\right)\right|=\left|\sum_{i=n}^{m} \phi_i(v_i)\right| \leq \sum_{i=n}^{m} |\phi_i(v_i)| \leq \sum_{i=n}^{m} \|\phi_i\|\,.
\end{align*}
Thus the result follows.
\end{proof}
\begin{rmk}
Consequently, we also have that
\begin{align*}
\left(\sum_{i=1}^{\infty} v_i \otimes \ev_i \right)^{\ast}=\left(\sum_{i=1}^{\infty} v_i^{\ast} \otimes \ev_i^{\ast}\right)
\end{align*}
is convergent in the weak${}^{\ast}$-topology.
\end{rmk}

\begin{prop}
    \label{prop:CuntzUnitary}
 Let $I=I_n$ or $I=I_{\infty}$. Then there is a one-to-one correspondence between right $M$-modular unitary maps 
    $U\colon L^2M_M\otimes \ell^2 I\to L^2M_M$ and families 
    $\{v_i\in M\}_{i\in I}$ that give representations of the Cuntz algebra 
    $\mathcal{O}_n$ or $\mathcal{O}_{\infty}$ 
    in $M$.
\end{prop}

\begin{proof}
Let $\{e_i: i \in I\}$ be an orthonormal basis of $\ell^{2}(I)$.
Note that for each $i$ the operator $|e_i \rangle \colon \CC \to \ell^2(I)$ is defined by
 \begin{align*}
 |e_i \rangle (\alpha)=\alpha e_i \qquad \alpha \in \CC\,,
 \end{align*}
 and the operator $\langle e_i|\colon \ell^2(I) \to \CC$ is defined by
 \begin{align*}
 \langle e_i|\eta=\langle \eta, e_i \rangle \qquad \eta \in \ell^{2}(I)\,.
 \end{align*}

 These two operators are adjoint to each other.
    Given a unitary $U \in \Hom_{-M}(L^2M_M \otimes \ell^2(I), L^2M_M)$, define 
    \begin{equation}
        v_i = U (\id_{L^2M_M} \otimes |e_i\rangle)r_{L^2M_M}^{\ast}\,.
    \end{equation} 
    
    Since $r_{L^2M_M}^{\ast}U \in \Hom_{-M}(L^2M_M \otimes \ell^2(I), L^2M_M \otimes \CC)$, we get
    \begin{align*}
    r_{L^2(M)_M}(m' \otimes \id_{\CC})r_{L^2M_M}^{\ast} U=U (m' \otimes \id_{\ell^2(I)})
    \end{align*}
    and consequently,
    \begin{align}\label{commutingU}
    m' U=U(m' \otimes \id_{\ell^2(I)}), \qquad m' \in M'\,.
    \end{align}
    For all $m' \in M'$ and $\xi \in L^2M_M$, we have 
    \begin{align*}
    (m' \otimes \id_{\CC}) r_{L^2M_M}^{\ast} \xi=(m' \otimes \id_{\CC}) (\xi \otimes 1_{\CC})=m' \xi \otimes 1_{\CC} \,.
    \end{align*}
    and 
    \begin{align*}
    r_{L^2M_M}^{\ast} m' \xi=r_{L^2M_M}^{\ast}(m' \xi)=m' \xi \otimes 1_{\CC}\,.
    \end{align*}
    For all $m' \in M'$, we get
    \begin{align}\label{Eqcuntz}
     (m' \otimes \id_{\CC}) r_{L^2M_M}^{\ast}=r_{L^2M_M}^{\ast} m' \,.
    \end{align}
   By using the definition of $v_i$ together with the Equation \eqref{commutingU} and the Equation~\eqref{Eqcuntz}, we get
    \begin{align*}
   m' v_i= m' U (\id_{L^2M_M} \otimes |e_i\rangle)r_{L^2M_M}^{\ast}&=U(m' \otimes \id_{\ell^2(I)})(\id_{L^2M_M} \otimes |e_i\rangle)r_{L^2M_M}^{\ast}\\
   &=U(m' \otimes  |e_i\rangle)r_{L^2M_M}^{\ast}\\
   &=U(\id_{L^2M_M} \otimes |e_i\rangle)(m' \otimes 1_{\CC})r_{L^2M_M}^{\ast}\\
   &=U(\id_{L^2M_M} \otimes |e_i\rangle)r_{L^2M_M}^{\ast} m' =v_i m' \qquad m' \in M'\,.
    \end{align*}
    Thus $v_i \in M$.
    In particular for $\xi \in L^2M_M$, 
    one has
    \begin{align*}
    v_i \xi= U (\xi \otimes e_i)\,.
    \end{align*}
    It suffices to check that $\{v_i: i \in I\}$ satisfy the Cuntz algebra relations.
    Note that for all $\xi, \eta \in L^2M_M$, we have
    \begin{align*}
    \langle v_i \xi, v_j \eta \rangle&=\langle U(\xi \otimes e_i), U(\eta \otimes e_j)\rangle\\
    &=\langle \xi \otimes e_i, \eta \otimes e_j \rangle \qquad (U \text{ is unitary})\\
    &=\langle \xi, \eta \rangle \langle e_i, e_j \rangle
    \end{align*}
    So we have $v_i^{\ast}v_i=1_M$ and $v_i^{\ast}v_j=0$ for $i \neq j$.
    We also have $v_i^{\ast}=r_{L^2M_M}(\id_{L^2M_M} \otimes \langle e_i|)U^{\ast}$.
    Notice that
    $$v_i v_i^{\ast}=U(\id_{L^2M_M} \otimes P_i)U^{\ast},$$ where $P_i=|e_i \rangle \langle e_i|$ is an orthogonal projection from $\ell^2(I)$ onto span$\{e_i\}$. 
    Therefore $\sum _{i} v_i v_i^{\ast}=U(\id_{L^2M_M} \otimes \sum_i P_i)U^{\ast}=U(\id_{L^2M_M} \otimes \id_{\ell^2(I)})U^{\ast}=1$.
    Hence $\{v_i: i \in I\}$ 
    fulfill the Cuntz algebra relations.
    Conversely, given a Cuntz family $\{v_i: i \in I\}$ in $M$, we define
    \begin{align}\label{unitarycun}
    U=r_{L^2M_M} \sum_{i} v_i \otimes  \langle e_i|\,.
    \end{align}
    Note that $U$ is well defined by Lemma~\ref{convseri}.
    Then notice that
    \begin{align*}
    U^{\ast} U= \sum_{i} \id_{L^2M_M} \otimes P_i=\id_{L^2M_M \otimes \ell^2(I)} \qquad(\text{the sum is convergent in the ambient topology})\,.
    \end{align*}
    Similarly observe that
    \begin{align*}
    UU^{\ast}&=r_{L^2(M)_M} (\sum_j v_j \otimes \langle e_j|)(\sum_{i} v_i^{\ast} \otimes |e_i \rangle)r_{L^2(M)_M}^{\ast}\\
    &=r_{L^2(M)_M}(\sum_i v_i v_i^{\ast} \otimes \id_{\CC})r_{L^2M_M}^{\ast}\\
    &=r_{L^2M_M}\left(\id_{L^2M_M} \otimes \id_{\CC}\right)r_{L^2M_M}^{\ast}=\id_{L^2(M)_M}\,.
    \end{align*}
    It suffices to check that
    \begin{align}\label{eqnabs}
    U(M \otimes \B(\ell^2I))U^{\ast}=M\,.
    \end{align}
    It suffices to check that Equation~\eqref{eqnabs} holds for any simple tensor in $M \otimes \B(\ell^2(I))$.
    Consider a simple tensor $m \otimes T \in M \otimes \B(\ell^2I)$. Applying Equation~\eqref{eqnabs}, for any $\xi \in L^2(M)_M$, we get
    \begin{align*}
    U(m \otimes T)U^{\ast} \xi&=r_{L^2(M)_M}\left(\sum_{i,j} v_i m v_j^{\ast} \otimes \langle e_i|T|e_j \rangle \right) r_{L^2(M)_M}^{\ast} \xi\\
    &=r_{L^2(M)_M}\left(\sum_{i,j} v_i m v_j^{\ast} \otimes \langle Te_j, e_i \rangle_{\ell^2(I)}\right) r_{L^2(M)_M}^{\ast} \xi\\
    &=r_{L^2(M)_M} \sum_{i,j} v_i m v_j^{\ast} \otimes \langle Te_j, e_i \rangle_{\ell^2(I)}(\xi \otimes 1_{\CC})\\
    &=r_{L^2(M)_M} \sum_{i,j} v_i m v_j^{\ast} \xi \otimes \langle Te_j, e_i \rangle_{\ell^2(I)}\\
    &=\sum_{i,j} \langle Te_j, e_i \rangle_{\ell^2(I)} v_i m v_j^{\ast} \xi\\
    &=\left(\sum_{i,j} \langle Te_j, e_i \rangle_{\ell^2(I)} v_i m v_j^{\ast}\right)\xi\,.
    \end{align*}
    Since $v_i, m\in M$, one has $\sum_{i,j} \langle Te_j, e_i \rangle_{\ell^2(I)} v_i m v_j^{\ast} \in M$, which concludes the proof.
\end{proof}

\section{Functor arising from Cuntz algebras and absorbing unitaries}
\label{sec:Cuntz functor}
Given a type $\mathrm{III}$ factor $M$ and Hilbert spaces $\Hil$ and $\K$, recall that
\begin{align*}
\End_{-M}(L^2M_M \otimes \Hil)=M \otimes \B(\Hil) \subseteq \B(L^2M_M \otimes \Hil)\,.
\end{align*}
If $\rho \in \End(M)$ then ${}^{\rho}(\cdot):=\rho \otimes \id_{\B(\Hil)} \in \End(M \otimes \B(\Hil))$, given by linear and normal extension of ${}^{\rho}(\cdot)$ on elementary tensors, namely
\begin{align*}
{}^{\rho}(m \otimes t):=\rho(m) \otimes t \qquad m \in M, ~t \in \B(\Hil)\,.
\end{align*}
Observe that $\Hom_{-M}(L^2M_M \otimes \Hil, L^2M_M \otimes \K)$ can be embedded in $\End_{-M}(L^2M_M \otimes (\Hil \oplus \K))$, identified with
\begin{align*}
\begin{pmatrix}
\End_{-M}(L^2M_M \otimes \Hil) & \Hom_{-M}(L^2M_M \otimes \K, L^2M_M \otimes \Hil) \\
\Hom_{-M}(L^2M_M \otimes \Hil, L^2M_M \otimes \K)  & \End_{-M}(L^2M_M \otimes \K)
\end{pmatrix}\,.
\end{align*}
More precisely, the embedding is given as
\begin{align*}
m \otimes s \mapsto \begin{pmatrix}
0 & 0 \\
m \otimes s & 0
\end{pmatrix} \qquad m \in M, s \in \B(\Hil, \cK)\,.
\end{align*}
On the other hand, the map in the other direction is given by
\begin{align*}
\begin{pmatrix}
a & b \\
c  & d
\end{pmatrix} \mapsto \tilde{V}_{\Hil} a \tilde{V}_{\Hil}^{\ast} + \tilde{V}_{\Hil} b \tilde{V}_{\K}^{\ast}+ \tilde{V}_{\K} c \tilde{V}_{\Hil}^{\ast}+ \tilde{V}_{\K} d \tilde{V}_{\K}^{\ast}\,,
\end{align*}
where $\tilde{V}_{\cdot}=1_M \otimes V_{\cdot}$.
Therefore we have
\begin{align*}
\Hom_{-M}(L^2M_M \otimes \Hil, L^2M_M \otimes \K) \hookrightarrow \End_{-M}(L^2M_M \otimes (\Hil \oplus \K))\,. 
\end{align*}
Note that the image of the aforementioned embedding is precisely
\begin{align*}
\{x \in \End_{-M}(L^2M_M \otimes (\Hil \oplus \K)): (1 \otimes p_{\K})x (1 \otimes p_{\Hil})=x\}\,,
\end{align*}
where $p_{\cdot}=V_{\cdot}V_{\cdot}^{\ast}$.
Observe that
\begin{align*}
{}^{\rho}((1 \otimes p_{\K}) x (1 \otimes p_{\Hil}))=(1 \otimes p_{\K}){}^{\rho}(x) (1 \otimes p_{\Hil})\,.
\end{align*}
Naturally define 
\begin{align}\label{endofn}
{}^{\rho}(x):=\tilde{V}_{\K}^{\ast}{}^{\rho}(\tilde{V}_{\K} x \tilde{V}_{\Hil}^{\ast})\tilde{V}_{\Hil} \quad x \in \Hom_{-M}(L^2M_M \otimes \Hil, L^2M_M \otimes \K)\,.
\end{align}
In particular, for $x \in \End_{-M}(L^2M_M)$,
one has
\begin{align}\label{specialcase}
{}^{\rho}(x)=\rho(x) \qquad \rho \in \End(M)\,.
\end{align}
Moreover, we also have 
\begin{align}\label{specialcase2}
{}^{\rho}(x)=r_{L^{2}M_M}{}^{\rho}(r_{L^2M_M}^{\ast}x) \quad x \in \Hom_{-M}(L^{2}M_M \otimes \Hil, L^2M_M)\,.
\end{align}
Notice that
\begin{align*}
{}^{\rho}(m \otimes s)={}^{\rho}(m) \otimes s \qquad m \in M, s \in \B(\Hil, \K)\,.
\end{align*}

It suffices to show that 
\begin{align*}
{}^\rho({}^{\sigma}(x))={}^{\rho \sigma}(x) \qquad \rho,~\sigma \in \End(M), \qquad x \in \Hom_{-M}(L^2M_M \otimes \Hil, L^2M_M \otimes \K)\,,
\end{align*} which will show that it is an endofunctor.
We compute 
\begin{align*}
{}^\rho({}^{\sigma}(x))&={}^{\rho}(\tilde{V}_{\K}^{\ast}{}^{\sigma}(\tilde{V}_{\K} x \tilde{V}_{\Hil}^{\ast})\tilde{V}_{\Hil})\\
&=\tilde{V}_{\K}^{\ast}{}^{\rho}(\tilde{V}_{\K}\tilde{V}_{\K}^{\ast}({}^{\sigma}(\tilde{V}_{\K}x \tilde{V}_{\Hil}^{\ast})\tilde{V}_{\Hil}\tilde{V}_{\Hil}^{\ast})\tilde{V}_{\Hil}\\
&=\tilde{V}_{\K}^{\ast}{}^{\rho}((1 \otimes p_{\K}){}^{\sigma}(\tilde{V}_{\K}x \tilde{V}_{\Hil}^{\ast})(1 \otimes p_{\Hil}))\tilde{V}_{\Hil}\\
&=\tilde{V}_{\K}^{\ast}{}^{\rho}({}^{\sigma}(\tilde{V}_{\K}x \tilde{V}_{\Hil}^{\ast}))\tilde{V}_{\Hil}\\
&=\tilde{V}_{\K}^{\ast}{}^{\rho \sigma}(\tilde{V}_{\K}x \tilde{V}_{\Hil}^{\ast})\tilde{V}_{\Hil} \qquad\text{(since the action is $\rho \sigma \otimes 1$ on $\Hom_{-M}(L^2M_M \otimes (\Hil \oplus \K))$)} \\
&={}^{\rho \sigma}(x) \qquad\text{(by Equation~\eqref{endofn})}\,.
\end{align*}
In particular, we want to show that
\begin{align*}
{}^{\rho}(xy)={}^{\rho}(x) {}^{\rho}(y)
\end{align*}
for any composable morphisms $x, y \in \Amp(M)$.
For this we need Roberts' $3$ by $3$ matrix trick.
Let $y \in \Hom_{-M}(-\Hil, -\K)$ and $x \in \Hom_{-M}(-\K, -\LL)$.
To avoid notational complexity, we write any right $M$-module $L^2M_M \otimes \Hil$ as $-\Hil$.
Then we have the embedding given by
\begin{align*}
\Hom_{-M}(-\Hil, -\K) \hookrightarrow 
& \End_{-M}(-(\Hil \oplus \K \oplus \LL))\\
& \cong
\begin{pmatrix}
\End_{-M}(-\Hil) & \Hom_{-M}(-\K,-\Hil) & \Hom_{-M}(-\LL,-\Hil)\\
\Hom_{-M}(-\Hil, -\K) & \End_{-M}(-\K) & \Hom_{-M}(-\LL, -\K)\\
\Hom_{-M}(-\Hil, -\LL) & \Hom_{-M}(-\K, -\LL) & \End_{-M}(-\LL)
\end{pmatrix}
\end{align*}
via the map
\begin{align}\label{Rob3by3}
\Hom_{-M}(-\Hil, -\K) \ni x \mapsto \tilde{V}_{\K} x \tilde{V}_{\Hil}^{\ast} \in \End_{-M}(-(\Hil \oplus \K \oplus \LL))\,,
\end{align}
where $V_1=V_{\Hil}, V_2=V_{\K}$, and $V_3=V_{\LL}$, and $\tilde V_{\cdot}=1_M \otimes V_{\cdot}$.
The image of the embedding~\eqref{Rob3by3} is precisely
\begin{align*}
\{x \in \End_{-M}(-(\Hil \oplus \K \oplus \LL)):(1 \otimes p_{\K})x(1 \otimes p_{\Hil})=x\}\,.
\end{align*}
Therefore by applying Equation~\eqref{endofn} via the embedding~\eqref{Rob3by3}, we get
\begin{align*}
&{}^{\rho}x {}^{\rho}y=\tilde{V}_{\LL}^{\ast}{}^{\rho}(\tilde{V}_{\LL} x \tilde{V}_{\K}^{\ast})\tilde{V}_{\K} \tilde{V}_{\K}^{\ast}{}^{\rho}(\tilde{V}_{\K} y \tilde{V}_{\Hil}^{\ast})\tilde{V}_{\Hil} \\
&=\tilde{V}_{\LL}^{\ast}({}^{\rho}\tilde{V}_{\LL} x \tilde{V}_{\K}^{\ast}{}^{\rho}\tilde{V}_{\K} \tilde{V}_{\K}^{\ast}{}^{\rho}\tilde{V}_{\K} y \tilde{V}_{\Hil}^{\ast}{}^{\rho}\tilde{V}_{\Hil}\tilde{V}_{\Hil}^{\ast})\tilde{V}_{\Hil}\\
&=\tilde{V}_{\LL}^{\ast}{}^{\rho}(\tilde{V}_{\LL} x \tilde{V}_{\K}^{\ast}\tilde{V}_{\K} \tilde{V}_{\K}^{\ast}\tilde{V}_{\K} y \tilde{V}_{\Hil}^{\ast}\tilde{V}_{\Hil}\tilde{V}_{\Hil}^{\ast})\tilde{V}_{\Hil}\\
&=\tilde{V}_{\LL}^{\ast}{}^{\rho}(\tilde{V}_{\LL} xy \tilde{V}_{\Hil}^{\ast})\tilde{V}_{\Hil} \qquad\text{(since ${}^{\rho}$ is a functor on $\End_{-M}(-(\Hil \oplus \K \oplus \LL))$)}\\
&={}^{\rho}(xy)\,.
\end{align*}
\begin{prop}\label{comp}
Let $M$ be a type $\mathrm{III}$ factor. Let $\Hil$ be a separable Hilbert space and $\sigma \colon M \to M \otimes \B(\Hil)$ is a morphism and $U \in \Hom_{-M}(L^2M_M \otimes \Hil,  L^2M_M)$ be a unitary.
Then $\rho=\Ad U \circ \sigma \in \End(M)$.
Then for any $\tau \in \End(M)$
\begin{align*}
\tau  \rho=\Ad {}^{\tau}U (\tau \otimes \id_{\B(\Hil)}) \sigma\,.
\end{align*}
\begin{proof}
Consider $r_{L^2M_M}^{\ast} U \in \Hom_{-M}(L^2M_M \otimes \Hil, L^2M_M \otimes \CC)$.
Recall from Equation~\eqref{specialcase2} that we have
\begin{align*}
{}^{\tau}U &=r_{L^{2}M_M} {}^{\tau} (r_{L^{2}M_M}^{\ast} U) \in \Hom_{-M}(L^2M_M \otimes \Hil, L^2M_M)\,.
\end{align*}
We write $r$ for $r_{L^{2}M_M}$ as shorthand notation.
Notice
\begin{align*}
r^{\ast}U \sigma(m) U^{\ast} r (m' \otimes \id_{\CC})(\xi \otimes 1_{\CC})&=r^{\ast}U \sigma(m) U^{\ast} m' \xi\\
&=\rho(m) m' \xi \otimes 1_{\CC}
\end{align*}
and
\begin{align*}
(m' \otimes \id_{\CC})r^{\ast}U \sigma(m) U^{\ast} r(\xi \otimes 1_{\CC})&=(m' \otimes \id_{\CC})r^{\ast}U \sigma(m) U^{\ast} \xi\\
&=(m' \otimes \id_{\CC})r^{\ast} \rho(m)\xi\\
&=(m' \otimes \id_{\CC})(\rho(m)\xi \otimes 1_{\CC})\\
&=m'\rho(m)\xi \otimes 1_{\CC}\,.
\end{align*}
Consequently, we note that
$r^{\ast}U \sigma(m) U^{\ast} r \in \End_{-M}(L^2M_M \otimes \CC)$ for each $m \in M$.
First note that
\begin{align*}
r (\tau \otimes \id_{\CC})(r^{\ast} U \sigma(m) U^{\ast} r) r^{\ast}\xi
&= r (\tau \otimes \id_{\CC})(r^{\ast} U \sigma(m) U^{\ast} \xi)\\
&= r (\tau \otimes \id_{\CC})(r^{\ast} \rho(m) \xi)\\
&= r (\tau \otimes \id_{\CC})(\rho(m) \xi \otimes 1_{\CC})\\
&=r(\tau(\rho(m)) \xi \otimes 1_{\CC})\\
&=\tau \rho (m) \xi \qquad \xi \in L^2M_M\,.
\end{align*}
Therefore, we get
\begin{align*}
\tau \rho(m)&=\Ad r (\tau \otimes \id_{\CC})(r^{\ast} U \sigma(m) U^{\ast} r)\\
&=\Ad r {}^{\tau}(r^{\ast}U \sigma(m) U^{\ast} r)\\
&=\Ad r {}^{\tau}(r^{\ast}U) {}^{\tau}(\sigma(m)) {}^{\tau}(U^{\ast} r) \qquad\text{(by functoriality of ${}^{\tau}(\cdot)$)}\\
&={}^{\tau}(U) {}^{\tau}(\sigma(m)) {}^{\tau}(U)^{\ast}\,,
\end{align*}
and hence
\begin{align}\label{comps1}
\tau \rho(m)={}^{\tau}(U) {}^{\tau}(\sigma(m)) {}^{\tau}(U)^{\ast}\,.
\end{align}
On the other hand, as $\sigma(m) \in \End_{-M}(L^{2}M_M \otimes \Hil)$, note that
\begin{align*}
(\tau \otimes \id_{\B(\Hil)})(\sigma(m))={}^{\tau}(\sigma(m))\,.
\end{align*}
As a result, we get
\begin{align}\label{composi2}
\Ad {}^{\tau}U (\tau \otimes \id_{\B(\Hil)}) \sigma(m)={}^{\tau}(U) {}^{\tau}(\sigma(m)) {}^{\tau}(U)^{\ast}\,.
\end{align}
Thus the left-hand side equals the right-hand side by Equations~\eqref{comps1} and~\eqref{composi2}, which concludes the proof.
\end{proof}
\end{prop}

\section{Minimal G-kernels}
\label{sec:minimal}
Let $G$ be a second countable compact group.
Consider a Borel lift $\alpha \colon G \to \Aut(M)$ of a $G$-kernel $\theta \colon G \to \Out(M)$, namely $[\alpha_g]=\theta_g$ for all $g \in G$.

Let $X=(X, \mu_X, g_X)$ be a triple consisting of a topological space $X$, a finite measure $\mu_X$, and a Borel map $g_X\colon X\to G$.

Consider the projection-valued measure $E_X=E_{\mu_X}$ (also known as a spectral measure) characterized by
$M_{f\circ g_X} = \int_X f(g_X(x)) \,\dd E_X(x)$ for $f\in C(G)$.
Define 
\begin{align}\label{partihom}
    \sigma_X(m) = \int_X \alpha_{g_X(x)}(m)\otimes \dd E_{X}(x)\,.
\end{align}
Then $\sigma_X\in \Hom(M, M \otimes \B(L^2(X)))$.
Because $M$ is type III, we can choose a unitary
$U_X\in \U_{-M}(L^2M_M\otimes L^2(X) , L^2M_M)$ (unitary right $M$-modular map) and define
\begin{align}\label{partiendo}
\rho_X = \Ad U_X \circ \sigma_X \in \End(M)\,.
\end{align}

In particular, let $\mu \in \Meas(G)$, where $\Meas(G)$ is the set of finite measures on $G$.
In this case, we have the triplet $(G, \mu, \id \colon G \to G)$.
Hence, we have
\begin{align*}
\sigma_G(m)=\int_G \alpha_{g}(m)\otimes \dd E_{G}(g)\,.
\end{align*}
We call a Borel subset $A$ of $G$ \emph{proper} if both $\mu(A)>0$ and $\mu(G \backslash A)>0$.
For a proper Borel subset $E \subset G$, define
\begin{align*}
\mu_E(A)=\mu(A \cap E), \qquad \text{for Borel measurable sets } A \subset G\,.
\end{align*}

\begin{prop}\label{Fullnesscdn}
Let $\mu \in \Meas(G)$ with $\alpha$ a lift of the $G$-kernel $\theta$. Then the following conditions are equivalent:
\begin{enumerate}
\item [{\rm (i)}] For any pair $E$ and $F$ of proper Borel subsets of $G$ with $\mu(E \cap F)=0$, the corresponding morphisms $\rho_{\mu_E}$ and $\rho_{\mu_F}$ admit no nontrivial intertwiners, equivalently, $(\rho_{\mu_E}, \rho_{\mu_F})=0$.
\item [{\rm (ii)}] For any proper Borel subset $E$ of $G$, the morphisms $\rho_{\mu_E}$ and $\rho_{\mu_{G\backslash E}}$ admit no nontrivial intertwiners, equivalently,
\begin{align*}
(\rho_{\mu_E}, \rho_{\mu_{G \backslash E}})=0\,.
\end{align*}
\end{enumerate}
\end{prop}
\begin{proof}
Assume condition $(i)$ holds; then by taking $F=G \backslash E$, condition $(ii)$ follows.

Conversely, assume condition $(ii)$ holds. Let $E, F$ be Borel measurable subsets such that $\mu(E \cap F)=0$; then $F \subseteq G \backslash E$ $\mu$-almost everywhere.
Here, by abuse of notation, we write $E$ instead of $\mu_E$.
Let $t \in (\rho_{E}, \rho_{F})$, then
\begin{align*}
&t U_E \sigma_E(m) U_E^{\ast}=U_F \sigma_F(m) U_F^{\ast} t\\
&U_F^{\ast} t U_E \sigma_E(m)=\sigma_F(m) U_F^{\ast} t U_E\,.
\end{align*}
Thus $\tilde{t}:=U_F^{\ast} t U_E \in (\sigma_E, \sigma_F)$.
Now note that $\Hil_E:=L^2(M) \otimes L^2(G, \mu_E)$ for any Borel measurable subset $E$ of $G$. Notice that $\sigma_E(M) \subseteq \B(\Hil_E)$.
Let $j$ be the isometry from $L^{2}(G, \mu_F)$ to $L^2(G, \mu_{G \backslash E})$.
Thus $\id_{L^2(M)} \otimes j$ is the embedding of $\Hil_F$ into $\Hil_{G \backslash E}$.
Then define $\hat{t}=(\id_{L^2(M)} \otimes j) \tilde{t}$.
Then $\hat{t} \in (\sigma_E, \sigma_{G \backslash E})=0$ because $(\rho_E, \rho_{G \backslash E})=0$ by assumption. Hence $t=0$, which concludes the proof.
\end{proof}

\begin{defi}\label{minG}
A $G$-kernel $\theta$ is said to be minimal if there exists a Borel lift $\alpha$ of $\theta$ that satisfies the conditions of Proposition~\ref{Fullnesscdn}.
\end{defi}
\begin{example}\label{minexmker}
\begin{enumerate}[label=(\roman*)]
\item For a finite group $G$, a $G$-kernel is minimal if and only if the $G$-kernel $\theta \colon G \ni g \mapsto [\alpha_g] \in \Out(M)$ is injective.
\item If a lift $\alpha \colon G \to \Aut(M)$ of a $G$-kernel is a minimal action, then the $G$-kernel is minimal.
\end{enumerate}
\end{example}

\subsection{Fullness of the functor}
\label{ssec:Fullness of the functor}
We consider the W${}^*$-category $\Rep_{\concrete}(C(G))$ of $C(G)$-modules that are separable Hilbert spaces.
Let $X=(X, \mu_X, g_X)$ denote a triple where $X$ is a topological space, $\mu_X$ is a finite measure, and $g_X\colon X\to G$ is a Borel map.

Then $L^2X$ becomes a $C(G)$-module via the action $f.\xi = M_{f\circ g_X}\xi$, where $f \in C(G)$ and $\xi \in L^2X$, and where $M_{\cdot}$ is the multiplication operator on $L^2X$.
In particular, $L^2X \in \Rep(C(G))$, and we regard $\Rep_{\concrete}(C(G))$ as a full subcategory of $\Rep(C(G))$.
Note that $\Rep(C(G))$ is a tensor category; see \cite{Ma2025}.
We define the map $F \colon \Rep_{\concrete}(C(G)) \to \End(M)$ on objects by
\begin{align*}
\Rep_{\concrete}(C(G)) \ni X \mapsto F(X)=\rho_X \in \End(M)
\end{align*}
and on morphisms by
\begin{align*}
F(t)=U_{Y} (1_M \otimes t) U_{X}^{\ast} \qquad t \in \Hom(X, Y)\,.
\end{align*}
Here $\rho_X$ is defined in Equation~\eqref{partiendo}.
\begin{thm}\label{Fullness}
Let $\theta \colon G \to \Out(M)$ be a $G$-kernel with a Borel lift $\alpha$.
Then the following conditions hold:
\begin{enumerate}[label=(\roman*)]
\item $F$ is a functor.
\item $F$ is faithful.
\item If $F$ is full, then the $G$-kernel is necessarily minimal.
\end{enumerate}
\end{thm}
\begin{proof}
Let $t_1 \in \Hom(X_1, X_2)$ and $t_2 \in \Hom(X_2, X_3)$.
Then $t_2t_1 \in \Hom(X_1, X_3)$ and therefore we have
\begin{align*}
F(t_2t_1) & =U_{X_3}(1_M \otimes t_2t_1)U_{X_1}^{\ast}\\
&=U_{X_3}(1_M \otimes t_2)U_{X_2}^{\ast} U_{X_2}(1_M \otimes t_1)U_{X_1}^{\ast}\\
&=F(t_2)F(t_1)\,.
\end{align*}
Furthermore $F(\id_X)=U_X(1_M \otimes \id_X)U_X^{\ast}=U_X\sigma_X(1_M)U_X^{\ast}=\rho_X(1_M)=\id_{\rho_X}=\mathrm{id}_{F(X)}$.
Hence $F$ is a functor.
For any $t \in \Hom(X, Y)$ such that $F(t)=0$, it follows that
\begin{align*}
U_{Y}(1_M \otimes t) U_X^{\ast}=0\,.
\end{align*}
Therefore, $t=0$, which shows that $F$ is faithful.
Finally, assume $F$ is full; then for any pair of proper Borel subsets $B, C$ of $G$, the map
\begin{align*}
F \colon \Hom(B, C) \to \Hom(\rho_B, \rho_C)
\end{align*}
is surjective, where $B$ can be regarded as an object in $\Rep_{\concrete}(C(G))$ as $(G,\mu_B,\id_{G})$.
If $\mu(B \cap C)=0$, then $\Hom(B, C)=0$.
Since $F$ is surjective, we have $\Hom(\rho_B, \rho_C)=0$.
Thus the $G$-kernel is minimal.
\end{proof}

\section{Tensor structure}
\label{sec:Tensor structure}
Let $G$ be a second countable compact group, $M \subseteq \B(L^2(M))$ a type $\mathrm{III}$ factor,
and $\alpha\colon G\to \Aut(M)$ a lift of a $G$-kernel $\theta \colon G \to \Out(M)$.
We fix a Borel map $u\colon G\times G \to U(M)$
such that
\begin{align*}
    \alpha_g \alpha_h = \Ad u_{g,h} \alpha_{gh} \qquad g,h \in G\,.
\end{align*}
 By the composition law of the endomorphisms applied to $\rho_X$ and $\rho_Y$ (see Proposition~\ref{comp}), we get
 \begin{align*}
 \rho_X  \rho_Y&=\Ad ({}^{\rho_X}U_Y) (\rho_X \otimes \id_{\B(L^{2}(Y))}) \sigma_Y\\
 &=\Ad ({}^{\rho_X}U_Y) ((\Ad U_X  \sigma_X) \otimes \id_{\B(L^{2}(Y))}) \sigma_Y\\
 &=\Ad ({}^{\rho_X}U_Y) (\Ad U_X \otimes \id_{\B(L^2(Y))}) (\sigma_X \otimes \id_{\B(L^{2}(Y))}) \sigma_Y\\
 &=\Ad ({}^{\rho_X}U_Y (U_X \otimes \id_{L^2(Y)}))  (\sigma_X \otimes \id_{\B(L^{2}(Y))}) \sigma_{Y}\,.
 \end{align*}
 Observe that 
 \begin{align*}
 (\sigma_X \otimes \id_{\B(L^{2}(Y))})  \sigma_Y(\cdot)&=(\sigma_X \otimes \id_{\B(L^{2}(Y))}) (\int_Y \alpha_{g_Y(y)}(\cdot) \otimes dE_{Y}(y))\\
 &=\int_Y \sigma_X(\alpha_{g_Y(y)}(\cdot)) \otimes dE_{Y}(y)\\
 &=\int_{Y} \int_{X} \alpha_{g_X(x)}(\alpha_{g_Y(y)}(\cdot)) \otimes dE_{X}(x) \otimes dE_{Y}(y)\\
 &=\int_{X \times Y} \alpha_{g_X(x)}  \alpha_{g_Y(y)}(\cdot) \otimes dE_{X}(x) \otimes dE_{Y}(y)\\
 &=\int_{X \times Y}  \Ad u_{g_X(x), g_Y(y)} \alpha_{g_X(x) g_Y(y)}(\cdot)\otimes dE_{X}(x) \otimes dE_{Y}(y)\,.
 \end{align*}
 Note that for the cartesian product $X \times Y$, the map is given by
 \begin{align*}
 g_{X \times Y}(x, y)=g_X(x) g_Y(y) \qquad (x, y) \in X \times Y\,.
 \end{align*}
 Define
 \begin{align*}
 u_{X, Y}=\int_{X \times Y} u_{g_X(x), g_Y(y)} \otimes dE_{X}(x) \otimes dE_{Y}(y)\,.
 \end{align*}
  Let $\iota_{X,Y} \colon L^2(X \times Y) \to L^2(X) \otimes L^2(Y)$ be the canonical isomorphism.
 We compute
 \begin{align*}
 &\Ad u_{X, Y} (\id_{L^2(M)} \otimes \iota_{X, Y}) \sigma_{X \times Y}(\cdot)\\
 &=\Ad u_{X, Y} (\id_{L^2(M)} \otimes \iota_{X, Y}) \int_{X \times Y} \alpha_{g_X(x)g_Y(y)}(\cdot) \otimes dE_{X \times Y}(x,y)\\
 &=\Ad u_{X, Y} \int_{X \times Y} \alpha_{g_X(x)g_Y(y)}(\cdot) \otimes dE_{X}(x) \otimes dE_{Y}(y)\\
 &=\left(\int_{X \times Y} u_{g_X(x), g_Y(y)} \otimes dE_{X}(x) \otimes dE_{Y}(y)\right) \left(\int_{X \times Y} \alpha_{g_X(x)g_Y(y)}(\cdot) \otimes dE_{X}(x) \otimes dE_{Y}(y)\right)\\
 & \quad \cdot \left(\int_{X \times Y} u_{g_X(x), g_Y(y)}^{\ast} \otimes dE_{X}(x) \otimes dE_{Y}(y)\right)\\
& =\int_{X^3 \times Y^3} u_{g_X(x_1), g_Y(y_1)} \alpha_{g_X(x_2)g_Y(y_2)}(\cdot)u_{g_X(x_3), g_Y(y_3)}^{\ast} \\
&\qquad \qquad \qquad\otimes dE_{X}(x_1)dE_{X}(x_2)dE_{X}(x_3) \otimes  dE_{Y}(y_1)dE_{Y}(y_2)dE_{Y}(y_3)\\
 &=\int_{X \times Y} u_{g_X(x), g_Y(y)} \alpha_{g_X(x)g_Y(y)}(\cdot) u_{g_X(x), g_Y(y)}^{\ast} \otimes  dE_{X}(x) \otimes  dE_{Y}(y)\\
 &=\int_{X \times Y} \Ad u_{g_X(x), g_Y(y)}  \alpha_{g_X(x)g_Y(y)}(\cdot)  \otimes  dE_{X}(x) \otimes  dE_{Y}(y)\,.
 \end{align*}
 Combining the above computations, we get
 \begin{align}
 (\sigma_X \otimes \id_{\B(L^{2}(Y))}) \sigma_Y(\cdot)=\Ad u_{X,Y} (\id_{L^2(M)} \otimes \iota_{X, Y})\sigma_{X \times Y}(\cdot)\,.
 \end{align}
 Moreover, we get
 \begin{align}
 \rho_X \rho_Y(\cdot)=\Ad ({}^{\rho_X}U_Y (U_X \otimes \id_{L^2(Y)}) u_{X, Y}  (\id_{L^2(M)} \otimes \iota_{X, Y})) \sigma_{X \times Y} (\cdot)\,.
 \end{align}
 In particular, one has
 \begin{align*}
 [\rho_X \rho_Y]=[\rho_{X \times Y}]\,.
 \end{align*}
 We define the tensorator by
 \begin{align}\label{tens}
 W_{X,Y}&=U_{X \times Y}(\id_{L^2(M)} \otimes \iota_{X,Y}^{\ast})u_{X,Y}^{\ast}({}^{\rho_X}U_Y (U_X \otimes \id_{L^{2}(Y)}))^{\ast}\\
 &=U_{X \times Y}(\id_{L^2(M)} \otimes \iota_{X,Y}^{\ast})u_{X,Y}^{\ast}(U_X^{\ast} \otimes \id_{L^{2}(Y)})({}^{\rho_X}U_Y)^{\ast}\,.
 \end{align}
 Since each term of $W_{X,Y}$ is unitary, $W_{X,Y}$ is unitary.
Note that
\begin{align*}
& W_{X, Y} (\rho_X \rho_Y)(\cdot)
=U_{X \times Y} (\id_{L^2(M)} \otimes \iota_{X,Y}^{\ast})u_{X,Y}^{\ast}(U_X^{\ast} \otimes \id_{L^{2}(Y)})\\
&{}^{\rho_X}U_Y^{\ast}{}^{\rho_X}U_Y (U_X \otimes \id_{L^{2}(Y)})u_{X,Y} (\id_{L^2(M)} \otimes \iota_{X,Y})\sigma_{X \times Y}(\cdot) ({}^{\rho_X}U_Y (U_X \otimes \id_{L^{2}(Y)})u_{X,Y} (\id_{L^2(M)} \otimes \iota_{X,Y}))^{\ast}\\
&=U_{X \times Y} \sigma_{X \times Y}(\cdot)(\id_{L^2(M)} \otimes \iota_{X,Y}^{\ast})u_{X,Y}^{\ast} (U_X^{\ast} \otimes \id_{L^{2}(Y)}) {}^{\rho_X}U_Y^{\ast}\\
&=U_{X \times Y} U_{X \times Y}^{\ast} \rho_{X \times Y}(\cdot)U_{X \times Y} (\id_{L^2(M)} \otimes \iota_{X,Y}^{\ast}) u_{X,Y}^{\ast} (U_X^{\ast} \otimes \id_{L^{2}(Y)}){}^{\rho_X}U_Y^{\ast}\\
&=\rho_{X \times Y}(\cdot)U_{X \times Y} (\id_{L^2(M)} \otimes \iota_{X,Y}^{\ast}) u_{X,Y}^{\ast} (U_X^{\ast} \otimes \id_{L^{2}(Y)}){}^{\rho_X}U_Y^{\ast}\\
&=\rho_{X \times Y}(\cdot) W_{X,Y}\,.
\end{align*}
Hence we get
\begin{align}\label{tenst}
W_{X, Y} \in \Hom_{\End(M)}(\rho_X \rho_Y, \rho_{X \times Y})\,.
\end{align}
Note that $W_{X, Y} \colon \rho_X \rho_Y \cong \rho_{X \times Y}$ is a natural isomorphism.
\subsection{Associator}
\label{ssec:Associator}
Given $\omega \in Z^{3}_{\Borel}(G, \TT)$, define the associator
\begin{align*}
[\alpha^{\omega}f](x, (y,z))=\omega(g_X(x), g_Y(y), g_Z(z))f((x, y), z)
\end{align*}
such that $\int |f((x, y),z)|^2 \, d E_{(X \times Y) \times Z}((x,y),z)=\int |f((x, y),z)|^2 \, d E_{X}(x)\,  dE_Y(y)\, dE_{Z}(z) < \infty$.
The $G$-kernel equation in the continuous case is precisely the following:

let $\iota_{X,Y} \colon L^2(X\times Y)\to L^2X \otimes L^2Y$ be the canonical isomorphism.
Since we need diagram commutation relation in order to show the pentagon diagram, we
consider the following computation in $\B(L^2(M) \otimes L^2X\otimes L^2Y\otimes L^2Z)$:
\begin{align}
    &(u_{X,Y}(\id_{L^2(M)}\otimes \iota_{X, Y})\otimes \id_{L^2Z})u_{X\times Y, Z}(\id_{L^2(M)}\otimes \iota^\ast_{X, Y}\otimes \id_{L^2Z}) \\
   \nonumber &\quad= \int_{X\times Y\times Z} u_{g_X(x),g_Y(y)} u_{g_X(x)g_Y(y), g_Z(z)} \otimes dE_X(x)\otimes dE_Y(y)\otimes dE_Z(z)\\
   \nonumber &\quad = \int_{X\times Y\times Z} \omega(g_X(x), g_Y(y), g_Z(z)) \alpha_{g_X(x)}(u_{g_Y(y),g_Z(z)})u_{g_X(x),g_Y(y)g_Z(z)} \\
 \nonumber   &\qquad \qquad \otimes dE_X(x)\otimes dE_Y(y)\otimes dE_Z(z)\\
    &\quad=(\id_{L^2(M)} \otimes \omega_{X,Y,Z}) \sigma_X(u_{Y, Z}) u_{X,Y\times Z}(\id_{L^2(M)}\otimes\id_{L^2X}\otimes \iota^\ast_{Y,Z})\,,
\end{align}
where
\begin{align}
    &\omega_{X,Y,Z} = \int\limits_{X\times Y\times Z} \omega(g_X(x), g_Y(y), g_Z(z))\,dE_X(x)\otimes dE_Y(y) \otimes dE_Z(z)\\
   \nonumber &\in \B(L^2(X)\otimes L^2(Y)\otimes L^2(Z)) \,.
\end{align}
\label{sec:Pentagon Diagram}
\appendix
\section{Pentagon Diagram}
In order to check that our functor $F$ is a unitary tensor functor, we first need an intermediate technical proposition.
\begin{prop}\label{tecintwin}
If $\rho$, $\sigma \in \End(M)$ and $\Hil$ is a separable Hilbert space such that $u \in \Hom(\rho, \sigma)$ and $U \in \U_{-M}(L^2M_M \otimes \Hil, L^2M_M)$, then
\begin{align*}
u {}^{\rho}U={}^{\sigma}U (u \otimes 1_{\B(\Hil)})\,.
\end{align*}
\end{prop}
\begin{proof}
Recall that ${}^{\rho}U=r_{L^2M_M}{}^{\rho}(r_{L^2M_M}^{\ast}U)$ and $r_{L^2M_M}^{\ast} U \in \U_{-M}(L^2M_M \otimes \Hil, L^2M_M \otimes \CC)$. By normality and linearity, it is enough to check the equation for $r_{L^2M_M}^{\ast} U=m \otimes t$ with $m \in M$ and $t \in \B(\Hil, \CC)$.
Note that for $\xi \in L^2M_M, \eta \in \Hil$, we have
\begin{align*}
u r_{L^2M_M} {}^{\rho} (m \otimes t)(\xi \otimes \eta)= u r_{L^2M_M}(\rho(m)\xi \otimes t\eta)=u (t \eta) (\rho(m) \xi)\qquad 
\end{align*}
and
\begin{align*}
r_{L^2M_M}{}^{\sigma}(m \otimes t) (u \otimes 1)(\xi \otimes \eta)&=r_{L^2M_M}(\sigma(m) u \otimes t)(\xi \otimes \eta)\\
&=r_{L^2M_M}(u \rho(m) \xi \otimes t \eta)=u (\rho(m) \xi) (t \eta)\,.
\end{align*}
Since $L^2M_M \odot \Hil$ is dense in $L^2M_M \otimes \Hil$, we get
\begin{align}\label{Eqnintwn2}
u r_{L^2M_M} {}^{\rho} (m \otimes t)=r_{L^2M_M}{}^{\sigma}(m \otimes t) (u \otimes 1_{\B(\Hil)})\,.
\end{align}
Hence the result follows by linearly extending Equation~\eqref{Eqnintwn2} and using the normality of the functor $\rho \mapsto {}^{\rho}$.
\end{proof}

Recall the formula given in Proposition~\ref{comp}: namely, given $\Hil$ a separable Hilbert space, any morphism $\sigma \colon M \to M \otimes \B(\Hil)$, and a unitary $U \in \Hom_{-M}(L^2M_M \otimes \Hil, L^2M_M)$,
then $\rho=\Ad U \circ \sigma \in \End(M)$.
Then for any $\tau \in \End(M)$, we have
\begin{align}\label{comeq}
\tau  \rho=\Ad {}^{\tau}U (\tau \otimes \id_{\B(\Hil)}) \sigma\,.
\end{align}
Using Equation~\eqref{comeq} on the endomorphisms $\rho_X$, $\rho_Y$ and $\rho_Z$ for $X, Y, Z \in \Rep_{\mathrm{concrete}}(C(G))$, we get
\begin{align*}
    (\rho_X \rho_Y) \rho_Z&=\Ad {}^{\rho_X \rho_Y}U_Z  (\rho_X \rho_Y \otimes \id_{\B(L^{2}(Z))}) \sigma_Z\\
    &=\Ad {}^{\rho_X \rho_Y}U_Z  ((\Ad {}^{\rho_X} U_Y  (\rho_X \otimes \id_{\B(L^{2}(Y))})  \sigma_Y) \otimes \id_{\B(L^{2}(Z))}) \sigma_Z\\
    &=\Ad {}^{\rho_X \rho_Y}U_Z ({}^{\rho_X} U_Y \otimes \id_{\B(L^{2}(Z))}) ((\rho_X \otimes \id_{\B(L^{2}(Y))}) \otimes \id_{\B(L^{2}(Z))}) (\sigma_Y \otimes \id_{\B(L^{2}(Z))})  \sigma_Z\\
    &=\Ad {}^{\rho_X \rho_Y}U_Z ({}^{\rho_X} U_Y \otimes \id_{\B(L^{2}(Z))}) (U_X \otimes \id_{L^{2}(Y)} \otimes \id_{L^{2}(Z)})\\
    &\qquad \qquad \qquad \qquad(\sigma_X \otimes \id_{\B(L^{2}(Y))} \otimes \id_{\B(L^{2}(Z))}) (\sigma_Y \otimes \id_{\B(L^{2}(Z))}) \sigma_Z\,.
\end{align*}
Therefore,
\begin{align*}
&{}^{\rho_X \rho_Y}U_Z ({}^{\rho_X} U_Y \otimes \id_{\B(L^{2}(Z))}) (U_X \otimes \id_{L^{2}(Y)} \otimes \id_{L^{2}(Z)}) \text{ is in} \\
& \Hom((\sigma_X \otimes \id_{\B(L^{2}(Y))} \otimes \id_{\B(L^{2}(Z))}) (\sigma_Y \otimes \id_{\B(L^{2}(Z))}) \sigma_Z, \rho_X \rho_Y \rho_Z)\,.
\end{align*}
By abuse of notation, we write $1_X=\id_{L^2(X)}$ for $X \in \Rep_{\mathrm{concrete}}(C(G))$ and $\id_X=\id_{\B(L^2(X))}$.

We recall the tensorator given as
 \begin{align}\label{tensgg}
 W_{X,Y}=U_{X \times Y}(\id_{L^2(M)} \otimes \iota_{X,Y}^{\ast})u_{X,Y}^{\ast}(U_X^{\ast} \otimes \id_{L^{2}(Y)})({}^{\rho_X}U_Y)^{\ast}\,.
 \end{align}
We want to show the following diagram commutes:

\begin{tikzcd}
\rho_X \rho_Y \rho_Z \arrow[dd, "{W_{X, Y}}"] &  &  &  &  & ((\sigma_X \otimes 1_Y) \otimes 1_Z)(\sigma_Y \otimes 1_Z)\sigma_Z \arrow[lllll, "{}^{\rho_X\rho_Y}U_Z({}^{\rho_X}U_Y \otimes 1_Z)(U_X \otimes 1_Y \otimes 1_Z)"] \arrow[dd, "{(\id_{L^2(M)} \otimes \iota_{X,Y}^{\ast} \otimes 1_Z) (u_{X,Y}^{\ast} \otimes 1_Z)}"] \\
                                              &  &  &  &  &                                                                                                                                                                                                              \\
\rho_{X \times Y}\rho_Z                       &  &  &  &  & (\sigma_{X \times Y} \otimes 1_Z)\sigma_Z \arrow[lllll, "({}^{\rho_{X \times Y}}U_Z)(U_{X \times Y} \otimes 1_Z)"]\,.
\end{tikzcd}

Note that
\begin{align}
    W_{X,Y}{}^{\rho_X \rho_Y} U_Z={}^{\rho_{X \times Y}}U_Z(W_{X,Y} \otimes 1_Z)
\end{align}
by Proposition~\ref{tecintwin}.
Thus, we get
\begin{align*}
    W_{X,Y}{}^{\rho_X \rho_Y}U_Z({}^{\rho_X}U_Y \otimes 1_Z)(U_X \otimes 1_Y \otimes 1_Z)&={}^{\rho_{X \times Y}}U_Z(W_{X,Y} \otimes 1_Z)({}^{\rho_X}U_Y \otimes 1_Z)(U_X \otimes 1_Y \otimes 1_Z)\\
    &={}^{\rho_{X \times Y}}U_Z(U_{X \times Y} \otimes 1_Z)(\id_{L^2(M)} \otimes \iota_{X,Y}^{\ast} \otimes 1_Z)\\
    & \qquad \qquad \quad (u_{X,Y}^{\ast} \otimes 1_Z)(U_X^{\ast} \otimes 1_Y \otimes 1_Z)\\
    & \qquad \qquad ({}^{\rho_X}U_Y^{\ast} \otimes 1_Z) ({}^{\rho_X}U_Y \otimes 1_Z)(U_X \otimes 1_Y \otimes 1_Z)\\
    &= {}^{\rho_{X \times Y}} U_Z(U_{X \times Y} \otimes 1_Z)(\id_{L^2(M)} \otimes \iota_{X,Y}^{\ast} \otimes 1_Z) (u_{X,Y}^{\ast} \otimes 1_Z)\\
    &={}^{\rho_{X \times Y}}U_Z(U_{X \times Y}\otimes 1_Z) (\id_{L^2(M)} \otimes \iota_{X,Y}^{\ast} \otimes 1_Z) (u_{X,Y}^{\ast} \otimes 1_Z)\,.
\end{align*}
Therefore the above diagram commutes.

Now we show that the following diagram commutes:
$$
\begin{tikzcd}
(\rho_X \otimes \mathrm{id}_{Y \times Z})\sigma_{Y \times Z} \arrow[dd, "{}^{\rho_X}U_{Y \times Z}"] &  &  &  & (\rho_X \otimes \mathrm{id}_Y \otimes \mathrm{id}_Z)(\sigma_Y \otimes \mathrm{id}_Z)\sigma_Z \arrow[llll, "{{}^{\rho_X}(\id_{L^2(M)} \otimes \iota_{Y,Z}^{\ast})u_{Y,Z}^{\ast}}"] \arrow[dd, "{}^{\rho_X \rho_Y}U_Z({}^{\rho_X}U_Y \otimes 1_Z)"] \\
                                                                                                        &  &  &  &                                                                                                                                                                                                                    \\
\rho_X \rho_{Y \times Z}                                                                                &  &  &  & \rho_X \rho_Y \rho_Z \arrow[llll, "{\rho_X(W_{Y,Z})}"]                                                                                                                                                          \,.
\end{tikzcd}
$$

Notice that
\begin{align*}
    {}^{\rho_X}u_{Y,Z}^{\ast}=\int \rho_X(u_{g_Y(y), g_Z(z)}^{\ast}) \otimes dE_{Y}(y) \otimes dE_{Z}(z)\,.
\end{align*}
Using the functor $\End(M) \ni \rho \mapsto {}^{\rho}(\cdot)$ on $\Amp(M)$, we get
\begin{align*}
    {}^{\rho_X}U_{Y \times Z}\,{}^{\rho_X}u_{Y,Z}^{\ast}= {}^{\rho_X}(U_{Y \times Z} u_{Y,Z}^{\ast})
\end{align*}
and
\begin{align*}
    \rho_X(W_{Y,Z}){}^{\rho_X \rho_Y}U_Z({}^{\rho_X}U_Y \otimes 1_Z)= \rho_X(W_{Y,Z}){}^{\rho_X}({}^{\rho_Y} U_Z)({}^{\rho_X}U_Y \otimes 1_Z)\,.
\end{align*}
Using the fact that ${}^{\rho}(\cdot)$ is an endofunctor on $\mathrm{Amp}(M)$, we also get
\begin{align*}
   ({}^{\rho_X} U_Y^{\ast} \otimes 1_Z)= {}^{\rho_X}(U_Y^{\ast} \otimes 1_Z)\,.
\end{align*}
Combining these and using the definition of $W_{X,Y}$, we observe
\begin{align*}
    &{}^{\rho_X}(U_{Y \times Z} (\id_{L^2(M)} \otimes \iota_{Y,Z}^{\ast})u_{Y,Z}^{\ast})({}^{\rho_X}U_Y \otimes 1_Z)^{\ast}({}^{\rho_X}({}^{\rho_Y} U_Z))^\ast\\
    &={}^{\rho_X}(U_{Y \times Z} (\id_{L^2(M)} \otimes \iota_{Y,Z}^{\ast}) u_{Y,Z}^{\ast}(U_Y^{\ast} \otimes 1_Z) {}^{\rho_Y}U_Z^{\ast})\\
&={}^{\rho_X}W_{Y,Z}\\&=\rho_X(W_{Y,Z})\,.
\end{align*}
Hence, we have
\begin{align*}
{}^{\rho_X}(U_{Y \times Z}(\id_{L^2(M)} \otimes \iota_{Y,Z}^{\ast})u_{Y,Z}^{\ast})=\rho_X(W_{Y,Z})\,{}^{\rho_X}({}^{\rho_Y} U_Z)({}^{\rho_X}U_Y \otimes 1_Z)\,.
\end{align*}
Thus, the diagram commutes.
Define
\begin{align*}
{}^{\sigma_X}(u_{Y,Z})=\int \sigma_X(u_{g_Y(y), g_Z(z)}) \otimes dE_Y(y) \otimes dE_Z(z)\,,
\end{align*}
and by definition of the endofunctor we have
\begin{align*}
{}^{\rho_X}(u_{Y,Z})=\int \rho_X(u_{g_Y(y), g_Z(z)}) \otimes dE_Y(y) \otimes dE_Z(z)\,.
\end{align*}
Since $\rho_X=\Ad U_X \circ  \sigma_X$, by linearity and continuity of the integral, we observe that
\begin{align*}
\Ad (U_X \otimes 1_Y \otimes 1_Z) {}^{\sigma_X}(u_{Y,Z}) = {}^{\rho_X}(u_{Y,Z})\,,
\end{align*}
and the same holds for
\begin{align*}
\Ad (U_X \otimes 1_Y \otimes 1_Z) {}^{\sigma_X}(u_{Y,Z}^{\ast}) = {}^{\rho_X}(u_{Y,Z}^{\ast})\,.
\end{align*}
Note that $\alpha_{g} \alpha_h=\Ad u_{g,h} \alpha_{gh}$ for all $g, h \in G$.
Note that
\begin{align*}
&(\sigma_X \otimes \id_Y \otimes \id_Z)(\sigma_Y \otimes \id_Z)\sigma_Z(\cdot)\\
&=(\sigma_X \otimes \id_Y \otimes \id_Z)(\sigma_Y \otimes \id_Z) \int_Z \alpha_{g_Z(z)}(\cdot) \otimes dE_Z(z)\\
&=(\sigma_X \otimes \id_Y \otimes \id_Z) \int_{Y \times Z} \alpha_{g_Y(y)} \alpha_{g_Z(z)} (\cdot) \otimes dE_{Y}(y) \otimes dE_Z(z)\\
&=\int_{X \times Y \times Z} \alpha_{g_X(x)}(\alpha_{g_Y(y)} \alpha_{g_Z(z)}(\cdot)) \otimes dE_X(x) \otimes dE_Y(y) \otimes dE_Z(z)\\
&=\int_{X \times Y \times Z} \alpha_{g_X(x)}(\Ad u_{g_Y(y), g_Z(z)} \alpha_{g_Y(y)g_Z(z)}(\cdot))\otimes dE_X(x) \otimes dE_Y(y) \otimes dE_Z(z)\\
&=\int \Ad (\alpha_{g_X(x)}(u_{g_Y(y), g_Z(z)})) \alpha_{g_X(x)}(\alpha_{g_Y(y)g_Z(z)}(\cdot)) \otimes dE_X(x) \otimes dE_Y(y) \otimes dE_Z(z)\\
&=\Ad ({}^{\sigma_X}(u_{Y,Z}) (\id_{L^2(M)} \otimes 1_X \otimes \iota_{Y,Z})) \int \alpha_{g_X(x)}(\alpha_{g_{Y \times Z}(y,z)}(\cdot)) \otimes dE_X \otimes dE_{Y \times Z}\\
&=\Ad({}^{\sigma_X}(u_{Y,Z})(\id_{L^2(M)} \otimes 1_X \otimes \iota_{Y,Z}))(\sigma_X \otimes \id_{Y \times Z})\sigma_{Y \times Z}\,.
\end{align*}
Thus we have
\begin{align*}
(\id_{L^2(M)} \otimes 1_X \otimes \iota_{Y,Z}^{\ast}) {}^{\sigma_X}(u_{Y,Z}^{\ast}) \in \Hom((\sigma_X \otimes \id_Y \otimes \id_Z)(\sigma_Y \otimes \id_Z)\sigma_Z, (\sigma_X \otimes \id_{Y \times Z})\sigma_{Y \times Z})\,.
\end{align*}
Note that
\begin{align}\label{imrq}
(U_X^{\ast} \otimes 1_{Y \times Z})(\id_{L^2(M)} \otimes \iota_{Y,Z}^{\ast})=(\id_{L^2(M)} \otimes 1_X \otimes \iota_{Y,Z}^{\ast})(U_X^{\ast} \otimes 1_Y \otimes 1_Z)\,. 
\end{align}
Using the fact that $\rho_X=\Ad U_X \circ \sigma_X$ and Equation~\eqref{imrq}, we have
\begin{align*}
&(U_X^{\ast} \otimes 1_{Y \times Z})(\id_{L^2(M)} \otimes \iota_{Y,Z}^{\ast}){}^{\rho_X}(u_{Y,Z}^{\ast})=(U_X^{\ast} \otimes 1_{Y \times Z})(\id_{L^2(M)} \otimes \iota_{Y,Z}^{\ast})\\
& \qquad  \qquad \qquad (U_X \otimes 1_Y \otimes 1_Z) {}^{\sigma_X}(u_{Y,Z}^{\ast})(U_X^{\ast} \otimes 1_Y \otimes 1_Z)\\
&=(\id_{L^2(M)} \otimes 1_X \otimes \iota_{Y,Z}^{\ast}){}^{\sigma_X}(u_{Y,Z}^{\ast})(U_X^{\ast} \otimes 1_Y \otimes 1_Z)\,.
\end{align*}
Hence the diagram:
$$
\begin{tikzcd}
(\sigma_X \otimes \mathrm{id}_{Y \times Z})\sigma_{Y \times Z} \arrow[dd, "U_X \otimes 1_{Y \times Z}"] &  &  &  & (\sigma_X \otimes \mathrm{id}_Y \otimes \mathrm{id}_Z)(\sigma_Y \otimes \mathrm{id}_Z)\sigma_Z \arrow[llll, "{(\id_{L^2(M)} \otimes 1_X \otimes \iota_{Y,Z}^{\ast}){}^{\sigma_X}u_{Y,Z}^{\ast}}"] \arrow[dd, "U_X \otimes 1_Y \otimes 1_Z"] \\
                                                                                                        &  &  &  &                                                                                                                                                                                        \\
(\rho_X \otimes \mathrm{id}_{Y \times Z})\sigma_{Y \times Z}                                            &  &  &  & (\rho_X \otimes \mathrm{id}_Y \otimes \mathrm{id}_Z)(\sigma_Y \otimes \mathrm{id}_Z)\sigma_Z \arrow[llll, "{(\id_{L^2(M)} \otimes \iota_{Y,Z}^{\ast}){}^{\rho_X}u_{Y,Z}^{\ast}}"]
\end{tikzcd}
$$
commutes.
Combining the results of each diagram is commutative we get
$$
\begin{tikzcd}
(\sigma_X \otimes \mathrm{id}_{Y \times Z})\sigma_{Y \times Z}
    \arrow[dd, "{}^{\rho_X}U_{Y \times Z} (U_X \otimes 1_{Y \times Z})"'] 
&  &  &  & 
(\sigma_X \otimes \mathrm{id}_Y \otimes \mathrm{id}_Z)
(\sigma_Y \otimes \mathrm{id}_Z)\sigma_Z
    \arrow[llll, "(\id_{L^2(M)} \otimes 1_X \otimes \iota_{Y,Z}^{\ast}){}^{\sigma_X}u_{Y,Z}^{\ast}"] 
    \arrow[dd, "{}^{\rho_X \rho_Y}U_Z({}^{\rho_X}U_Y \otimes 1_Z)(U_X \otimes 1_Y \otimes 1_Z)"] 
\\
& & & & \\
\rho_X \rho_{Y \times Z} 
&  &  &  & 
\rho_X \rho_Y \rho_Z 
    \arrow[llll, "\rho_X(W_{Y,Z})"'] 
\end{tikzcd}
$$
is commutative.

Recall $\omega_{X,Y,Z} \in \B(L^2(X) \otimes L^2(Y) \otimes L^2(Z))$.

Then we get
\begin{align*}
&\tilde{\omega}_{X,Y,Z}^{\ast}:=\id_{L^2(M)} \otimes (\iota_{X, Y \times Z}^{\ast}(1_X \otimes \iota_{Y,Z}^{\ast})\omega_{X,Y,Z}^{\ast}\\
& \qquad \qquad \qquad(\iota_{X,Y} \otimes 1_Z)\iota_{X \times Y, Z}) \quad \text{in}\\
& \Hom(L^2(M) \otimes L^2((X \times Y) \times Z), L^2(M) \otimes L^2(X \times (Y \times Z)))\,.
\end{align*}
By definition of our functor, we get
\begin{align*}
F(a_{X,Y,Z})=U_{X \times (Y \times Z)} \tilde{\omega}_{X,Y,Z}^{\ast} U_{(X \times Y) \times Z}^{\ast}\,.
\end{align*}
Therefore
$$
\begin{tikzcd}
\sigma_{(X \times Y) \times Z} \arrow[rrr, "{\tilde{\omega}_{X,Y,Z}^{\ast}}"] \arrow[dd, "U_{(X \times Y) \times Z}"] &  &  & \sigma_{X \times (Y \times Z)} \arrow[dd, "U_{X \times (Y \times Z)}"] \\
                                                                                                      &  &  &                                                                        \\
\rho_{(X \times Y) \times Z} \arrow[rrr, "{F(a_{X,Y,Z})}"]                                            &  &  & \rho_{X \times (Y \times Z)}                                          
\end{tikzcd}
$$
is commutative.
Note that because of the $G$-kernel relations for the sectors the following diagram also commutes:
$$
\begin{tikzcd}
((\sigma_X \otimes \mathrm{id}_Y) \otimes \mathrm{id}_Z)(\sigma_Y \otimes \mathrm{id}_Z)\sigma_Z \arrow[rrrr, "\mathrm{id}", two heads] \arrow[d, "{(\id_{L^2(M)} \otimes \iota_{X,Y}^{\ast} \otimes 1_Z)(u_{X,Y}^{\ast} \otimes 1_Z)}"] &  &  &  & ((\sigma_X \otimes \mathrm{id}_Y) \otimes \mathrm{id}_Z)(\sigma_Y \otimes \mathrm{id}_Z)\sigma_Z \arrow[d, "{(\id_{L^2(M)} \otimes \id_{L^2(X)} \otimes \iota_{Y,Z}^{\ast}){}^{\sigma_X}u_{Y,Z}^{\ast}}" swap] \\
(\sigma_{X \times Y} \otimes \mathrm{id}_Z)\sigma_Z \arrow[d, "{(\id_{L^2(M)} \otimes \iota_{X \times Y, Z}^{\ast})u_{X \times Y, Z}^{\ast}}"]                                                                           &  &  &  & (\sigma_X \otimes \mathrm{id}_{Y \times Z})\sigma_{Y \times Z} \arrow[d, "{(\id_{L^2(M)} \otimes \id_{L^2(X)} \otimes \iota_{Y,Z}^{\ast})u_{X, Y \times Z}^{\ast}}" swap]                                   \\
\sigma_{(X \times Y) \times Z}   \arrow[rrrr, "{\tilde{\omega}_{X,Y,Z}^{\ast}}"]                                                                                                 &  &  &  & \sigma_{X \times (Y \times Z)}                                                                                                   \,.  
\end{tikzcd}
$$

Here \begin{align*}
\tilde{\omega}_{X,Y,Z}^{\ast}&=(\id_{L^2(M)} \otimes \iota_{X, Y \times Z}^{\ast})(\id_{L^2(M)} \otimes 1_X \otimes \iota_{Y,Z}^{\ast})\\&\qquad \qquad(\id_{L^2(M)} \otimes \omega^{\ast}_{X,Y,Z})(\id_{L^2(M)} \otimes \iota_{X,Y} \otimes 1_Z)(\id_{L^2(M)} \otimes \iota_{X \times Y, Z})\,.
\end{align*}
Note that
\begin{align*}
    \omega_{X,Y,Z}^{\ast}=\int_{X \times Y \times Z}\overline{\omega(g_X(x),g_Y(y),g_Z(z))} \otimes dE_X(x) \otimes dE_Y(y) \otimes dE_Z(z)\,.
\end{align*}
Note also that
$$
\begin{tikzcd}
\rho_{X \times Y} \rho_Z \arrow[dd, "{W_{X \times Y, Z}}"] &  &  &  & (\sigma_{X \times Y} \otimes \mathrm{id}_Z)\sigma_Z \arrow[llll, "{}^{\rho_{X \times Y}}U_Z(U_{X \times Y} \otimes 1_Z)"] \arrow[dd, "{(\id_{L^2(M)} \otimes \iota_{X \times Y, Z}^{\ast})u_{X \times Y,Z}^{\ast}}"] \\
                                                             &  &  &  &                                                                                                                                                                   \\
\rho_{(X \times Y) \times Z}                                 &  &  &  & \sigma_{(X \times Y) \times Z} \arrow[llll, "U_{(X \times Y) \times Z}"]                                                                                         
\end{tikzcd}
$$
is commutative by definition of $W_{X \times Y,Z}$.
Moreover, also note that following diagram commutes:
$$
\begin{tikzcd}
(\sigma_X \otimes \id_{Y \times Z})\sigma_{Y \times Z} \arrow[rrr, "{}^{\rho_X}U_{Y \times Z}(U_X \otimes 1_{Y \times Z})"] \arrow[dd, "{(\id_{L^2(M)} \otimes \iota_{X, Y \times Z}^{\ast})u_{X, Y \times Z}^{\ast}}"] &  &  & \rho_X \rho_{Y \times Z} \arrow[dd, "{W_{X, Y \times Z}}"] \\
                                                                                                                                                                     &  &  &                                                              \\
\sigma_{X \times (Y \times Z)} \arrow[rrr, "U_{X \times (Y \times Z)}"]                                                                                              &  &  & \rho_{X \times (Y \times Z)}                                
\end{tikzcd}
$$
by definition of $W_{X, Y\times Z}$.

Consequently, the full pentagon diagram commutes:
\begin{center}
\resizebox{\textwidth}{!}{
\begin{tikzcd}[ampersand replacement=\&, column sep=small, row sep=4em]
\rho_X \rho_Y \rho_Z 
    \arrow[rrrr, two heads, "\id"] 
    \arrow[dd, "W_{X,Y}"'] 
\& \& \& \& 
\rho_X \rho_Y \rho_Z 
    \arrow[dd, "\rho_X(W_{Y,Z})"'] 
\\
\& ((\sigma_X \otimes \id_Y) \otimes \id_Z)(\sigma_Y \otimes \id_Z)\sigma_Z 
    \arrow[rr, two heads, "\id"] 
    \arrow[lu, "b"] 
    \arrow[d, "(\id_{L^2(M)} \otimes \iota_{X,Y}^{\ast} \otimes 1_Z)(u_{X,Y}^{\ast} \otimes 1_Z)"'] 
\& \& 
    ((\sigma_X \otimes \id_Y) \otimes \id_Z)(\sigma_Y \otimes \id_Z)\sigma_Z 
    \arrow[ru, "b"]
    \arrow[d, "c'"] 
\& \\
\rho_{X \times Y}\rho_Z 
    \arrow[dd, "W_{X \times Y, Z}"'] 
  \& (\sigma_{X \times Y} \otimes \id_Z)\sigma_Z 
    \arrow[d, "(\id_{L^2(M)} \otimes \iota_{X \times Y, Z}^{\ast})u_{X \times Y,Z}^{\ast}"'] 
    \arrow[l, "{}^{\rho_{X\times Y}}U_{Z}(U_{X \times Y} \otimes 1_Z)"] 
\& \& 
    (\sigma_X \otimes \id_{Y \times Z})\sigma_{Y \times Z} 
    \arrow[d, "(\id_{L^2(M)} \otimes \iota_{X, Y \times Z}^{\ast})u_{X, Y \times Z}^{\ast}"] 
    \arrow[r, "{}^{\rho_X}U_{Y \times Z}(U_X \otimes 1_{Y \times Z})"] 
  \& \rho_X \rho_{Y \times Z} 
    \arrow[dd, "W_{X, Y \times Z}"'] 
\\
\& \sigma_{(X \times Y)\times Z} 
    \arrow[rr, "\tilde{\omega}_{X,Y,Z}^{\ast}"] 
    \arrow[ld, "U_{(X \times Y) \times Z}"'] 
\& \& 
    \sigma_{X \times (Y \times Z)} 
    \arrow[rd, "U_{X \times (Y \times Z)}"] 
\& \\
\rho_{(X \times Y) \times Z} 
    \arrow[rrrr, "F(a_{X,Y,Z})"] 
\& \& \& \& 
\rho_{X \times (Y \times Z)}
\end{tikzcd}
}
\end{center}

where $b={}^{\rho_X\rho_Y}U_Z({}^{\rho_X}U_Y \otimes 1_Z)(U_X \otimes 1_Y \otimes 1_Z)$, and $c'=(\id_{L^2(M)} \otimes \id_{L^2(X)} \otimes \iota_{Y,Z}^{\ast}){}^{\sigma_X}u_{Y,Z}^{\ast}$.

Therefore the required diagram commutes, namely:

$$
\begin{tikzcd}
\rho_X \rho_Y \rho_Z 
  \arrow[rrrr, "\id"] 
  \arrow[d, "W_{X,Y}"'] 
& & & & 
\rho_X \rho_Y \rho_Z 
  \arrow[d, "\rho_{X}(W_{Y,Z})"] 
\\
\rho_{X \times Y} \rho_Z 
  \arrow[d, "W_{X \times Y, Z}"'] 
& & & & 
\rho_X \rho_{Y \times Z} 
  \arrow[d, "W_{X, Y \times Z}"] 
\\
\rho_{(X \times Y) \times Z} 
  \arrow[rrrr, "F(a_{X,Y,Z})"'] 
& & & & 
\rho_{X \times (Y \times Z)}
\end{tikzcd}
$$

commutes.

\bibliographystyle{plain}

\begin{bibdiv}
\begin{biblist}

\bib{BiKaLoRe2014-2}{book}{
      author={Bischoff, Marcel},
      author={Kawahigashi, Yasuyuki},
      author={Longo, Roberto},
      author={Rehren, Karl-Henning},
       title={Tensor categories and endomorphisms of von {N}eumann
  algebras---with applications to quantum field theory},
      series={Springer Briefs in Mathematical Physics},
   publisher={Springer},
        date={2015},
      volume={3},
        ISBN={978-3-319-14300-2; 978-3-319-14301-9},
         url={http://dx.doi.org/10.1007/978-3-319-14301-9},
      review={\MR{3308880}},
}

\bib{DoHaRo1969II}{article}{
      author={Doplicher, Sergio},
      author={Haag, Rudolf},
      author={Roberts, John~E.},
       title={{Fields, observables and gauge transformations II}},
        date={1969},
        ISSN={0010-3616},
     journal={Comm. Math. Phys.},
      volume={15},
       pages={173–200},
         url={http://dx.doi.org/10.1007/BF01645674},
}

\bib{EtGeNiOs2015}{book}{
      author={Etingof, Pavel},
      author={Gelaki, Shlomo},
      author={Nikshych, Dmitri},
      author={Ostrik, Victor},
       title={Tensor categories},
      series={Mathematical Surveys and Monographs},
   publisher={American Mathematical Society, Providence, RI},
        date={2015},
      volume={205},
        ISBN={978-1-4704-2024-6},
         url={http://dx.doi.org/10.1090/surv/205},
      review={\MR{3242743}},
}

\bib{GhLiRo1985}{article}{
      author={Ghez, P.},
      author={Lima, R.},
      author={Roberts, J.~E.},
       title={{$W^\ast$}-categories},
        date={1985},
        ISSN={0030-8730},
     journal={Pacific J. Math.},
      volume={120},
      number={1},
       pages={79\ndash 109},
  url={http://projecteuclid.org.proxy.library.ohio.edu/euclid.pjm/1102703884},
      review={\MR{808930}},
}

\bib{Ha1975}{article}{
      author={Haagerup, Uffe},
       title={The standard form of von {N}eumann algebras},
        date={1975},
     journal={Math. Scand.},
      volume={37},
      number={2},
       pages={271\ndash 283},
}

\bib{HeNiPe2024}{misc}{
      author={Henriques, André},
      author={Nivedita},
      author={Penneys, David},
       title={Complete w*-categories},
        date={2024},
         url={https://arxiv.org/abs/2411.01678},
}

\bib{Iz2015}{article}{
      author={Izumi, Masaki},
       title={A {C}untz algebra approach to the classification of near-group
  categories},
        date={2015},
     journal={arXiv preprint arXiv:1512.04288},
}

\bib{Jo1980}{article}{
      author={Jones, Vaughan F.~R.},
       title={Actions of finite groups on the hyperfinite type {${\rm
  II}_{1}$}\ factor},
        date={1980},
        ISSN={0065-9266},
     journal={Mem. Amer. Math. Soc.},
      volume={28},
      number={237},
       pages={v+70},
         url={http://dx.doi.org/10.1090/memo/0237},
      review={\MR{587749}},
}

\bib{Ma2025}{misc}{
      author={Marín-Salvador, Adrià},
       title={Continuous tambara-yamagami tensor categories},
        date={2025},
         url={https://arxiv.org/abs/2503.14596},
}

\bib{Su1980-II}{article}{
      author={Sutherland, Colin~E.},
       title={Cohomology and extensions of von {N}eumann algebras. {II}},
        date={1980},
        ISSN={0034-5318},
     journal={Publ. Res. Inst. Math. Sci.},
      volume={16},
      number={1},
       pages={135\ndash 174},
         url={https://doi.org/10.2977/prims/1195187501},
      review={\MR{574031}},
}

\end{biblist}
\end{bibdiv}
\def\cprime{$'$}\newcommand{\noopsort}[1]{}

\end{document}